\documentclass[11pt]{article}

\usepackage{amsmath,amssymb,amsfonts,enumitem,amsthm,accents,pstricks,pst-node,pstricks-add,mathrsfs,xy,graphicx}
\xyoption{all}
\usepackage{hyperref}

\numberwithin{equation}{section}

\newtheorem{theorem}{Theorem}[section]
\newtheorem{lemma}[theorem]{Lemma}
\newtheorem{conjlemma}[theorem]{Conjectural Lemma}
\newtheorem{corollary}[theorem]{Corollary}
\newtheorem{proposition}[theorem]{Proposition}
\newtheorem{conjecture}[theorem]{Conjecture}

\theoremstyle{definition}
\newtheorem{definition}[theorem]{Definition}
\newtheorem{remark}[theorem]{Remark}
\newtheorem{example}[theorem]{Example}

\def\ov#1{\overline{#1}}
\def\tn#1{\textnormal{#1}}
\def\mf#1{\mathfrak{#1}}
\def\wt#1{\widetilde{#1}}

\def\vr{\varrho}
\def\ll{\left\langle}
\def\rr{\right\rangle}
\def\mc{\mathcal}
\def\lra{\longrightarrow}
\def\dbar{\bar\partial}

\def\bEq#1{\begin{equation}\label{#1}}
\def\eEq{\end{equation}}

\def\bsEq{\begin{equation*}}
\def\esEq{\end{equation*}}

\def\bDf#1{\begin{definition}\label{#1}}
\def\eDf{\end{definition}}

\def\bTh#1{\begin{theorem}\label{#1}}
\def\eTh{\end{theorem}}

\def\bCn#1{\begin{conjecture}\label{#1}}
\def\eCn{\end{conjecture}}

\def\bLm#1{\begin{lemma}\label{#1}}
\def\eLm{\end{lemma}}

\def\bCLm#1{\begin{conjlemma}\label{#1}}
\def\eCLm{\end{conjlemma}}

\def\bRm#1{\begin{remark}\label{#1}}
\def\eRm{\end{remark}}

\def\bEx#1{\begin{example}\label{#1}}
\def\eEx{\end{example}}

\def\bPr#1{\begin{proposition}\label{#1}}
\def\ePr{\end{proposition}}

\def\bCr#1{\begin{corollary}\label{#1}}
\def\eCr{\end{corollary}}

\def\bFg#1{\begin{figure}\label{#1}}
\def\eFg{\end{figure}}

\def\bPf{\begin{proof}}
\def\ePf{\end{proof}}

\def\bIt{\begin{itemize}[leftmargin=*]}
\def\eIt{\end{itemize}}

\def\bEn{\begin{enumerate}[label=$(\arabic*)$,leftmargin=*]}
\def\eEn{\end{enumerate}}

\def\bEnalph{\begin{enumerate}[label=$(\alph*)$,leftmargin=*]}
\def\eEnalph{\end{enumerate}}

\def\AK{\tn{AK}}
\def\coker{\tn{coker}}

\def\nd{\tn{d}}

\def\Def{\tn{Def}}
\def\ev{\tn{ev}}

\def\ker{\tn{ker}}

\def\Ob{\tn{Ob}}

\def\Symp{\tn{Symp}}

\def\cC{\mc{C}}

\def\cM{\mc{M}}
\def\cO{\mc{O}}
\def\cE{\mc{E}}

\def\cP{\mc{P}}

\def\cN{\mc{N}}
\def\cR{\mc{R}}

\def\cC{\mc{C}}
\def\cH{\mc{H}}

\def\cT{\mc{T}}

\def\E{\mathbb E}
\def\R{\mathbb R}
\def\C{\mathbb C}
\def\Z{\mathbb Z}
\def\Q{\mathbb Q}
\def\P{\mathbb P}

\def\N{\mathbb N}

\def\mfi{\mf{i}}
\def\mfj{\mf{j}}

\def\mfs{\mf{s}}

\def\la{\lambda}

\def\De{\Delta}
\def\de{\delta}
\def\om{\omega}
\def\Om{\Omega}

\def\Si{\Sigma}
\def\al{\alpha}

\def\ze{\zeta}

\begin{document}

\title{BPS invariants of \\ symplectic log Calabi-Yau fourfolds}
\author{Mohammad Farajzadeh-Tehrani}
\date{\today}
\maketitle

\begin{abstract}
Using the Fredholm setup of \cite{FT2}, we study genus zero (and higher) relative Gromov-Witten invariants with maximum tangency of symplectic log Calabi-Yau fourfolds. In particular, we give a short proof of  \cite[Cnj.~6.2]{GPS} that expresses these invariants in terms of certain integral invariants by considering generic almost complex structures to obtain a geometric count. 
We also revisit the localization calculation of the multiple-cover contributions in \cite[Prp.~6.1]{GPS} and recalculate a few terms differently to provide more details and illustrate the computation of deformation/obstruction spaces for maps that have components in a destabilizing (or rubber) component of the target.
Finally, we study a higher genus version of these invariants and explain a decomposition of genus one invariants into different contributions.
\end{abstract}

\tableofcontents

\section{Introduction}\label{intro_s}
Suppose $(X,\om)$ is a smooth closed symplectic manifold, $D\!\subset\!X$ is a smooth symplectic divisor, and $J$ is an $\om$-tame almost complex structure on $X$ preserving $TD$ (i.e., $JTD\!=\!TD$). Then, for every genus $g$ $J$-holomorphic map $u\colon\! \Si\! \lra\! X$ with a smooth domain that is not mapped into $D$ and represents the homology class $A\!\in\! H_2(X,\Z)$, there is a finite (possibly empty) set of positive integers
$$
\mfs\equiv (s_1,...,s_k) \in \Z_+^k, \qquad \sum_{i=1}^k s_i =A\cdot D\geq 0,
$$
such that $u^{-1}(D) \!=\! \{z_1,\ldots,z_k\}\!\subset\! \Si$ and $u$ has a well-defined tangency of order $s_a$ at $z_a$ with $D$ as in the holomorphic case. We usually require $z_a$ to be (ordered) marked points on $\Si$. Furthermore, we may consider extra classical marked points $z_a$ on $\Si$ and define $s_a=0$ for them. Then,  $u^{-1}(D) \!\subset \! \{z_1,\ldots,z_k\}$ and $\mfs \in \N^k$. This way, we do not need to distinguish between the classical and contact marked points as each non-negative integer~$s_a$ determines the type of the marked point~$z_a$.  \\

\noindent
Two marked $J$-holomorphic maps 
$$
\big(u,C\equiv (\Si,z_1,\ldots,z_k))\big) \quad \tn{and}\quad \big(u',C'\equiv (\Si',z_1',\ldots,z_k')\big)
$$ 
are equivalent if there exists a holomorphic identification $h\colon C\stackrel{\cong}{\lra} C'$ of the marked domains $C$ and $C'$ such that $u=u'\circ h$. Following the terminology of \cite{FT1}, we denote the space of the equivalence classes of $k$-marked genus $g$ $J$-holomorphic maps of contact type $\mfs$ with $D$ by $\cM_{g,\mfs}(X,D,A).$ 
By \cite[(1.7)]{FT2}, the real expected dimension of $\cM_{g,\mfs}(X,D,A)$ is  
\bEq{dlog_e}
2\Big(\ll c_1(TX(-\log D)),A\rr+(1-g)(\dim_\C X -3)+k\Big),
\eEq
where $TX(-\log D)$ is the logarithmic tangent bundle constructed in \cite{FMZ2}. In Section~\ref{DO_e}, we will review the construction and properties of the  logarithmic tangent bundle $TX(-\log D)$.

\begin{remark}
For any finite set $S$, with $\mfs \in (\N^S)^k$ instead of $\N^k$, the definition of the moduli space $\cM_{g,\mfs}(X,D,A)$ above and the formula  (\ref{dlog_e}) naturally extend to the case where $D=\bigcup_{i\in S} D_i$ is a simple normal crossings symplectic divisor in the sense of \cite[Dfn.~2]{FMZ1}; see \cite{FT1}. However, the classical relative compactification does not generalize easily. There are different approaches in the algebraic and symplectic categories to construct a well-behaved compact moduli space; for instance see \cite{AC,C,GS, FT1, I, P, R}.
\end{remark}

\begin{definition}
We say 
\bIt
\item $\big[u,\Si,(z_1,\ldots,z_k)\big]\in \cM_{g,\mfs}(X,D,A)$ has \textbf{maximal tangency} with $D$ if 
$$
s_1\!=\!s_A\!\equiv A\cdot D\!>\!0\qquad (\tn{and}~s_i=0\quad \forall~i>1);
$$
\item $((X,\om),D)$ is a  \textbf{symplectic log Calabi-Yau pair} if 
\bEq{logCY_e}
c_1(TX(-\log D))=c_1(TX)-\tn{PD}_X(D)=0,
\eEq
where $\tn{PD}_X(D)\!\in\!H^2(X,\Z)$ is the Poincare dual of $D$ in $X$;
\item $((X,\om),D)$ is a symplectic log Calabi-Yau (or \textbf{CY}) \textbf{fourfold} if $\dim_\R X \!=\!4$. 
\eIt
\end{definition}

\noindent
Let $\tn{Symp}_{\log}(X,D)$ denote the space of symplectic structures $\om$ on $X$ such that $D$ is a symplectic submanifold and $((X,\om),D)$ is a symplectic log CY pair. 
To drop $\om$ from the notation, we say $(X,D)$ is a symplectic log CY pair if $\tn{Symp}_{\log}(X,D)\neq \emptyset$.\\

\noindent
If $(X,D)$ is a symplectic log CY fourfold, by (\ref{logCY_e}) and the adjunction formula, $D$ will be a symplectic $2$-torus.  For the main result of this paper, we are interested in the case where $(X,D)$ is a symplectic log Calabi-Yau fourfold, $\Si$ has genus $0$ (i.e., $\Si\!=\!\P^1$) and one marked point $z_1$, and $u$ has maximal tangency with $D$ at $z_1$. In other words, we study the invariants arising from (the relative compactification of) $\cM_{0,(s_A)}(X,D,A)$. In Section~\ref{Gen_s}, we study an invariant arising from $\cM_{g,(s_A)}(X,D,A)$ with $g>0$.\\

\noindent
The moduli space $\cM_{g,(s_A)}(X,D,A)$ is often not compact. We need more conditions on $J$ along $D$ to construct a compactification; see \cite[Sec.~1]{FT1}. For $(\om,J)$ in a suitable space $\tn{AK}(X,D)$ of almost K\"ahler structures as in \cite[Thm.~1.4~or~Rmk.~1.5]{FT1}, we can use the relative compactification $\ov\cM^{\tn{rel}}_{g,(s_A)}(X,D,A)$ studied in \cite{IP1,LR,Li, FZ1} or the (analytic) log compactification $\ov\cM^{\log}_{g,(s_A)}(X,D,A)$ constructed in \cite{FT1}. 
The latter is slightly smaller than the relative compactification  in the sense that there is a surjective map 
$$
\ov\cM^{\tn{rel}}_{g,(s_A)}(X,D,A)\lra \ov\cM^{\log}_{g,(s_A)}(X,D,A);
$$ 
see \cite[Prp~4.5]{FT1} or Appendix~\ref{RC_A}. Furthermore, if $g\!=\!0$, by \cite[Thm.~1.4]{FT1}, the natural forgetful map 
$$
\ov\cM^{\log}_{0,(s_A)}(X,D,A)\lra \ov\cM_{0,1}(X,A)$$
into the underlying stable map compactification is a topological embedding\,\footnote{Here, $1$ denotes the number of marked points}. The log compactification $\ov\cM^{\log}_{g,\mfs}(X,D,A)$ in \cite{FT1} is expected to be related to the algebraic log compactification of Gross-Siebert and Abramovich-Chen \cite{GS,AC,C}. Nevertheless, \cite[Cnj.~6.2]{GPS} and this paper  concern Gromov-Witten invariants arising from the relative compactification.

\begin{remark}
In Theorem~\ref{Compactification_th} below and several other places, the particular choice of the compactification is not important. Therefore, for simplicity of the notation, when this is the case, we drop the superscript and denote the compactified moduli space by $\ov\cM_{0,(s_A)}(X,D,A)$.  
\end{remark}

\noindent
A stable $1$-marked genus zero $J$-holomorphic map $(u,(\Si,z_1))$ with a genus zero nodal domain $\Si$ is \textbf{simple}\footnote{Note that $(\Si,z_1)$ is a unstable; c.f. \cite[Dfn.~3.4]{FT2} for the definition of simple maps in general.} if the restriction $u_v$ of $u$ to each irreducible component $\Si_v$ of $\Si$ is not multiply-covered (or equally it is somewhere injective) whenever $u_v$ is not constant, and the images of two such components in $X$ are distinct; see \cite[Sec.~2.5]{MS}. An element of $\ov\cM_{0,(s_A)}(X,D,A)$ is called simple if the underlying stable map in $\ov\cM_{0,1}(X,A)$  is simple. Let 
$$
\ov\cM^{\star}_{0,(s_A)}(X,D,A)\!\subset\! \ov\cM_{0,(s_A)}(X,D,A)
$$ 
denote the subspace of simple log/relative curves. Since every log CY fourfold $(X,D)$ is semi-positive in the sense of \cite[Dfn.~4.7]{FZ1} or equally \cite[Dfn.~1.6]{FT2}, the following theorem is mainly a corollary of \cite[Thm.~1.5(2)]{FT2}, \cite[Prp.~1.7(2)]{FT2}, Lemma~\ref{LItoSI_lmm}, and some classical results. We will explain the proof with details at the end of Section~\ref{DO_e}. 

\begin{theorem}\label{Compactification_th}
Suppose $(X,D)$ is a symplectic log Calabi-Yau fourfold. Then, for every $E\!>\!0$, there exists a Baire subset $\tn{AK}^{\tn{reg}}(X,D)\!\subset \!\!\tn{AK}(X,D)$ such that the forgetful map 
$$
\tn{AK}^{\tn{reg}}(X,D) \lra \tn{Symp}_{\log}(X,D), \quad (\om,J)\lra \om,
$$
has connected fibers and,
for every $(\om,J)\!\in\! \tn{AK}^{\tn{reg}}(X,D)$ and $A\!\in\! H_2(X,\Z)$ with $\om(A)<E$,
\bIt
\item the moduli space $\cM^{\star}_{0,(s_A)}(X,D,A)$ is cut transversely and consists of finitely many points\footnote{i.e., it has no accumulation point in $\ov\cM^{\star}_{0,(s_A)}(X,D,A)$};
\item every log/relative curve in $\ov\cM^{\star}_{0,(s_A)}(X,D,A)$ has  a smooth domain, i.e.,
\bEq{LogSpace_e}
\ov\cM^{\star}_{0,(s_A)}(X,D,A)=\cM^{\star}_{0,(s_A)}(X,D,A);
\eEq
\item and,  each multiple-cover map in 
$$
\cM^{\tn{mc}}_{0,(s_A)}(X,D,A)\equiv \ov\cM_{0,(s_A)}(X,D,A)-\ov\cM^{\star}_{0,(s_A)}(X,D,A)$$
has an irreducible image in $X$. In particular, $\ov\cM_{0,(s_A)}(X,D,A)$ is a disjoint union of closed components
$$
\ov\cM_{0,(s_A)}(X,D,A)\cong \coprod_{\substack{B\in H_2(X,\Z)\\ dB=A}}\ov\cM_{0,(d)}\big(\P^1,\infty,[d \P^1]\big)^*\times \cM^\star_{0,(s_B)}(X,D,B).
$$

\eIt
 
\end{theorem}
\noindent
We will explain the meaning of the superscript $*$ on $\ov\cM_{0,(d)}\big(\P^1,\infty,[d \P^1]\big)^*$ in Remark~\ref{s_B-rmk}.\\

\noindent
In fact, we show that for every arbitrary almost K\"ahler structure $(\om,J)$, each element of  the moduli space $\cM_{0,(s_A)}(X,D,A)$  is automatically \textbf{super-rigid} in the sense that $\cM_{0,(s_A)}(X,D,A)$ does not have a sequence of maps with distinct images accumulating at a multiple-cover map; see \cite[Dfn.~2.3]{W2}. 
This statement mainly follows from \cite[Thm.~1']{HLS} that concerns the injectivity of all $\R$-linear $\dbar$-operator on a line bundle of negative degree. 
With some effort, Theorem~\ref{Compactification_th} can also be extracted from the automatic transversality results of Wendl in \cite{W1}, by viewing $X\!-\!D$ as a manifold with a cylindrical end and working with punctured surfaces. However, our logarithmic Fredholm setup has the advantage of working with closed surfaces and is more concrete for applications such as localization calculations. 
The super-rigidity statement above is similar but more straightforward than the Wendl's super-rigidity theorem \cite[Thm.~A]{W2} for symplectic Calabi-Yau sixfolds that holds for some choices of $J$. 
In particular, the connectivity of $\tn{AK}^{\tn{reg}}(X,D)$  in Thm~\ref{Compactification_th} does not hold for symplectic Calabi-Yau sixfolds, resulting in a wall-crossing phenomenon where embedded $J$-holomorphic curves bifurcate when passing through certain real codimension $1$ walls in the space of almost complex structures; see \cite{W2, BS}. 
As we explain in Section~\ref{Gen_s}, the analogy above between the genus zero integral curve counts in Calabi-Yau sixfolds and symplectic log Calabi-Yau fourfolds does not naturally extend to higher genus. \\

 \noindent
For every 1-marked curve $[u,\P^1,z_1]\in \cM_{0,(s_A)}(X,D,A)$, after a reparametrization of the domain, we may assume that the contact point with $D$ happens at $z_1\!=\!\infty\!\in\! \P^1$. If $u$ is a degree $d$ multiple-cover map, we have
\bEq{mc-decomposition}
u=\ov{u}\circ h\colon \P^1=\C\cup \{\infty\} \lra X
\eEq
such that  $h(z)$ is a degree $d$ polynomial in $z$, and $\ov{u}\colon\! \P^1\!\lra\! X$ is a degree $B$ somewhere-injective $J$-holomorphic map with 
$$
A=dB\quad \tn{and} \quad \ov{u}^{-1}(D)=\infty \in \P^1.
$$ 
 The tuple $\big(\ov{u},(\P^1,\infty)\big)$ represents an element of 
$\cM_{0,(s_B)}^\star(X,D,B)$ and  
the tuple $(h,(\P^1,\infty))$ represents an element of the open-dense relative/log moduli space $\cM_{0,(d)}\big(\P^1,\infty,[d \P^1]\big)$. 
In the limit, a sequence of degree $d$ relative holomorphic maps in $\cM_{0,(d)}\big(\P^1,\infty,[d \P^1]\big)$
may converge to a relative/log map $h'$  with nodal domain. In the relative compactification, $h'$ is a map
\bEq{Relative-lim_e}
h'\colon (\Si,z_1) \lra \P^1[\ell]\equiv  \underbrace{\P^1 {}_{\infty}\!\!\cup_{0} \P^1 \cdots {}_{\infty}\!\!\cup_{0} (\P^1,\infty)}_{\ell+1~\tn{copies of}~\P^1 }
\eEq
with image in a chain (or expanded degeneration) of rational curves such that the component containing the marked point $z_1$ is mapped to the last $\P^1$ and $z_1$ is mapped to $\infty$. Nevertheless, the image of $u'\!=\!\ov{u}\circ h'$ will remain to be the same irreducible curve $\tn{image}(\ov{u})$. 

\begin{remark}\label{s_B-rmk}
The relative and log compactifications of $\cM_{0,(d)}\big(\P^1,\infty,[d \P^1]\big)$ are different and the difference can potentially\footnote{The multiple-cover localization calculation is done using the relative compactification.} result in different Gromov-Witten invariants (i.e., different multiple-cover contributions).
The set of relative maps $(h',(\Si,z_1))$ in (\ref{Relative-lim_e}) can be topologically identified with $\ov\cM^{\tn{rel}}_{0,(d)}(\P^1,\infty,[d \P^1])$. However, if $s_B\!>\!1$, as mentioned in \cite[p.~352]{GPS}, the algebraic moduli structure or the orbifold structure somewhat differs from the standard structure  on $\ov\cM_{0,(d)}^{\tn{rel}}\big(\P^1,\infty,[d \P^1]\big)$ and depends on $s_B$. 
The reason is that the $s_B$-tangency of $\ov{u}$ with $D$ forces the tangency orders of maps to the rubber components of $\P^1[\ell]$ to all be divisible by $s_B$; c.f. Appendix~\ref{RC_A} for a quick review of the relative compactification. In \cite{GPS}, the correct moduli/orbifold structure is indicated by a superscript $*$; i.e., the moduli space is denoted by $\ov\cM^{\tn{rel}}_{0,(d)}(\P^1,\infty,[d \P^1])^*$. We will explain these technical details in Section~\ref{loc_sec}.
\end{remark}

\bTh{GW_th}
Suppose $(X,D)$ is a symplectic log Calabi-Yau fourfold. 
The relative Gromov-Witten invariants 
\bEq{G0NA_e}
N_A\equiv \# [\ov\cM^{\tn{rel}}_{0,(s_A)}(X,D,A)]^{\tn{VFC}} \in \Q\qquad \forall~A\!\in\!H_2(X,\Z),~s_A>0,
\eEq
are defined\,\footnote{i.e., they can be defined without using sophisticated virtual techniques.} and only depend on the deformation equivalence class of $\om\in \Symp_{\log}(X,D)$.
Furthermore, 
\bEq{mc-formula_e}
N_A=\sum_{\substack{B\in H_2(X,\Z)\\ dB=A, d>0}}\tn{mc}(d,s_B)~n_{B}
\eEq
such that $n_{B}\!\in\!\Z$ for all $0\neq B\!\in\!H_2(X,\Z)$ and
\bEq{mc_e}
\tn{mc}(d,s_B)=d^{-2}{d(s_B -1)-1 \choose d-1}.
\eEq
For generic $J$, $n_A$ is the count of ``logarithmically immersed" curves in $\cM^\star_{0,(s_A)}(X,D,A)$.
\eTh
\noindent
The notion of \textbf{logarithmically immersed} curve is defined in Section~\ref{DO_e}.
If $s_B=1$, note that 
$$
\tn{mc}(d,1)=d^{-2}{-1\choose d-1}=\frac{(-1)^{d-1}}{d^{2}}.
$$
\vskip.1in

\noindent
Similarly to the BPS or Gopakumar-Vafa formula for the genus zero Gromov-Witten (or GW) invariants of symplectic CY sixfolds, \cite[Cnj.~6.2]{GPS} predicts that by taking into account the contribution of multiple-cover maps we obtain an integral invariant which we call the genus zero relative/log BPS invariants of $(X,D)$. As the last statement of Theorem~\ref{Compactification_th} specifies, the proof shows that these BPS numbers can be realized as a geometric count of finitely many logarithmically-rigid curves.\\

\noindent
In \cite[Prp.~6.1]{GPS}, the coefficients $\tn{mc}(d,s_B)$ are calculated using the relative localization calculation. 
As mentioned in \cite[p.~352]{GPS}, the moduli space $\ov\cM^{\tn{rel}}_{0,(d)}(\P^1,\infty,[d \P^1])^*$ is a nonsingular Deligne-Mumford stack or complex orbifold of complex dimension $d-1$.
The contribution $\tn{mc}(d,s_B)$ of $d$-fold multiple-covers of a somewhere injective degree $B$ map $\ov{u}$ to~$N_{dB}$ is 
\bEq{LocInt_e}
\tn{mc}(d,s_B)=\int_{\ov\cM^{\tn{rel}}_{0,(d)}\big(\P^1,\infty,[d \P^1]\big)^*} c_{\tn{top}}\big(\tn{Ob}(d,s_B)\big),
\eEq
where $\tn{Ob}(d,s_B)$ is the obstruction bundle of rank $d-1$. Recall from Remark~\ref{s_B-rmk} that the relative moduli space of multiple covers of $\P^1$ and thus the obstruction bundle depends on $s_B$. In order to compute (\ref{LocInt_e}), one needs to explicitly describe $\tn{Ob}(d,s_B)$ and choose a suitable lift of the natural $S^1$-action on $\P^1$ to the (logarithmic) normal bundle of $\ov{u}$. The latter determines the lift of the action to $\tn{Ob}(d,s_B)$.
When $s_B=1$, this is done in the proof of \cite[Thm. 5.1]{BP}. For $s_B>1$, the proper choice of the lift is provided in the proof of \cite[Prp.~6.1]{GPS} and the justification is mostly left to reader. In the second half of Section~\ref{loc_sec}, we review the efficient argument of \cite{GPS} and describe an alternative approach that involves different weights\footnote{It corresponds to infinitesimally extending the action to a neighborhood of the contributing curve in $X$.} and is closer to the realm of  \cite{GV}. While the latter approach is computationally cumbersome, it illustrates how the analytic description of the deformation-obstruction spaces in this paper works for relative maps that have components in the rubber components of the target.    \\

\noindent
From the algebraic perspective,  \cite[Cnj.~6.2]{GPS} is proved for toric del Pezzo surfaces in \cite{GWZ}. There are several other  recent works such as \cite{CGKT} and \cite{GGR} that address \cite[Cnj.~6.2]{GPS} or a variation of that from the algebraic or tropical perspectives.  The (expected) relation of these invariants to certain count of $J$-holomorphic disks in $X\!-\!D$, Mirror Symmetry, loop quiver  Donaldson-Thomas invariants, and local CY threefold invariants are explained in \cite{CGKT}, \cite{BBG}, and other recent works. We do not know of any other work that approaches these invariants from a purely symplectic/analytic perspective.  \\

\noindent
In Section~\ref{Gen_s}, following \cite{B2}, we will discuss a natural higher genus version of the relative invariants (\ref{G0NA_e}) and provide an explicit decomposition of the genus one invariants into contributions of genus-zero curves and conjecturally integral ``reduced" genus one invariants.

\begin{remark}
In the algebraic geometry literature, it is shown that the genus $0$ relative GW invariants $N_A$ and their higher genus variants are explicitly related to the GW invariants of the local CY threefold $K_{X}$, where $K_X$ is the total space of the canonical bundle of $X$; see~\cite{BFGW}. The integral invariants can then be defined using the Gopakumar-Vafa formula for CY threefolds. Nevertheless, it is shown in \cite[Lmm.~12]{GWZ} that these integral invariants are related by an integral matrix with an integral inverse that has a natural interpretation in terms of Donaldson-Thomas invariants of loop quivers. Therefore, the integrality of the two definitions of BPS invariants are equivalent.
\end{remark}

\begin{example}
Let $X=\P^2$, $D$ be a cubic curve, $\ell$ denote the line class in $H_2(\P^2,\Z)$. The first row of the table below\footnote{Borrowed from \cite[Table~7.1-7.2]{G}.  Gathmann's Growi software calculates these invariants in arbitrary (computationally feasible) degree.} contains the relative GW invariants $N_d\equiv N_{d\ell}$ for $1\leq d \leq 6$ and the second row contains the BPS numbers $n_d$ in the same degree range.
\begin{center}
\begin{tabular}{ |c|c|c|c|c|c|c|} 
  \hline
 $d$ &$1$ & $2$ & $3$ & $4$ & $5$ & $6$  \\
  \hline
   $N_d$ &$9$ & $135/4$ & $244$ & $36999/16$ & $635634/25$ & $307095$ \\
  \hline
   $n_d $ &$9$ & $27$ & $234$ & $2232$ & $25380$ & $305829$ \\
   \hline
\end{tabular}
\end{center}
For instance, the number $n_2$ of maximally tangent conics can be calculated in the following way. Let $C$ be a conic in $\P^2$ that is maximally tangent to the cubic curve $D$ at a point $P\in D$. Then, 
$$
\cO_{D}(6P)= \cO(2)|_{D},
$$
where $\cO_{D}(6P)$ is the degree $6$ line bundle on the elliptic curve $D$ corresponding to the divisor $6P$ and $\cO(2)$ is the degree $2$ line bundle on $\P^2$. Thus, the set of such points $P$ is the fiber over $\cO(2)|_{D}\in \tn{Pic}^6(D)$ of the map
$$
D\cong \tn{Pic}^1(D) \lra \tn{Pic}^6(D), \qquad L \lra L^{\otimes 6},
$$
between the degree $1$ and $6$ Picard groups.
The latter is a 36:1 covering map, so there are $36$ such points in the pre-image of $\cO(2)|_{D}$. Moreover, for each of these points, the totally tangent conic is unique. However, for a similar reason, $9$ of these $36$ points correspond to double of maximally tangent lines. Therefore, $n_2=27$.
\end{example}

\begin{remark}
As the last calculation indicates, for the curves contributing to $N_A$, the tangent points with $D$ have finite order in the elliptic curve $D$. Therefore, in the holomorphic setting, it is possible to refine the invariants $N_A$ and $n_A$ using the order of the contact points in $D$, as is done in  \cite{CGKT} and \cite{B1}. Since $\tn{dim}_\C D=1$, every $J\in \AK(X,D)$ is integrable in a neighborhood of $D$ in~$X$.  Therefore, such a refinement of $N_A$ and $n_A$ also seems possible from the analytic perspective.  
\end{remark}

\noindent
I am thankful to B.~Pierrick, J.~Choi, M.~Swaminathan, and T.~Gr\"afnitz, M.~Gross,~R.~Pandeharipande,~T.~Graber,~R.~Vakil, and Sasha on MathOverFlow for answering my questions. This work is supported by the NSF grant DMS-2003340.

\section{Logarithmic Fredholm theory}\label{DO_e}
In this section, we review the Fredholm setup for the deformation-obstruction theory of relative/log maps in the virtually main stratum $\cM_{g,\mfs}(X,D,A)$ introduced in \cite{FT2}. The main output of this discussion will be the notion of logarithmic normal bundle (\ref{LogNorBundle_e}) and the super-rigidity Lemma~\ref{LItoSI_lmm}.
We conclude this section with a relatively short proof of  Theorem~\ref{Compactification_th}.\\

\noindent
Let us start with a brief discussion of the classical Fredholm setup. Suppose $\Si$ is a smooth closed  Riemann surface with the complex structure $\mfj$. Let $\tn{Map}_{A}(\Si,X)$ denote the space of all smooth maps $u\colon\!\Si\!\lra\!X$ that represent the homology class~$A$, and  
$$
\cE_A(\Si,X)\lra\tn{Map}_{A}(\Si,X)
$$
be the infinite dimensional bundle whose fiber over $u$ is $\Gamma(\Si, \Om^{0,1}_{\Si}\otimes_{\C} u^*TX)$. The Cauchy-Riemann (or \textbf{CR}) equation 
\bEq{Jnu-holo_e}
\dbar u\equiv \frac{1}{2}\big( \nd u+ J \nd u \circ \mfj)
\eEq 
can be seen as a section $\dbar\colon\tn{Map}_{A}(\Si,X)\! \lra\! \cE_A(\Si,X)$. More precisely, we consider a Sobolev completion of these spaces for the Implicit Function Theorem to apply, but, by elliptic regularity, every solution of $\dbar u \!=\!0$ will be smooth; see \cite[Appendix~B]{MS}.  The linearization of the $\dbar$-section at any $J$-holomorphic map $u$ is an $\R$-linear map
\bEq{Ddbar_e}
\tn{D}_u \dbar\colon \Gamma(\Si,u^*TX)\lra \Gamma(\Si,\Om^{0,1}_{\Si}\otimes_\C u^*TX)
\eEq
that is the sum of a $\C$-linear $\dbar$-operator and a compact operator. Therefore, it is a Fredholm operator and Riemann-Roch applies; i.e., it has finite dimensional kernel and co-kernel, and
\bEq{RR_e}
\dim_\R \Def(u)\!-\!\dim_\R \Ob(u)\!=\!2\big( \tn{deg}(u^*TX) \!+\! \dim_\C\! X (1\!-\!g)\big),
\eEq
where 
$$
\Def(u)\!=\!\ker(\tn{D}_u\dbar)\quad \tn{and}\quad   \Ob(u)\!=\!\coker(\tn{D}_u\dbar).
$$
The first space corresponds to infinitesimal deformations of $u$ (over the fixed smooth marked domain) and the second one is the obstruction space for integrating elements of $\Def(u)$ to actual deformations. 
If $\Ob(u)\!=\! 0$, by Implicit Function Theorem \cite[Thm.~A.3.3]{MS}, in a small neighborhood $U$ of $u$ in $\tn{Map}_A(\Si,X)$, the set of $J$-holomorphic maps $V=U\cap \dbar^{-1}(0)$ is a smooth manifold of real dimension (\ref{RR_e}), all the elements of $\Def(u)$ are smooth, and $T_u V\cong \Def(u)$;
see \cite[Thm.~3.1.5]{MS}. The manifold $V$ carries a natural orientation.\\

\noindent
For $(\om,J)\!\in\! \tn{AK}(X,D)$, in \cite[Sec.~4]{FT2}, corresponding to every genus $g$ $k$-marked relative map in the virtually main stratum, 
\bEq{GE_e}
f\!=\![u,C=(\Si,z_1,\ldots,z_k)]\in \cM_{g,\mfs}(X,D,A),
\eEq
we derive a \textbf{logarithmic linearization} of the CR operator/section at $u$, denoted by
\bEq{Dlog_e}
\tn{D}^{\log}_u\dbar\colon\Gamma(\Si,u^*TX(-\log D))\lra \Gamma(\Si,\Om^{0,1}_{\Si}\otimes_\C u^*TX(-\log D)),
\eEq
such that $TX(-\log D)$ is the logarithmic tangent bundle introduced in \cite{FMZ2} and constructed in detail in \cite{FMZ4}.
After the following remark, we briefly digress from the main discussion and recall the construction of $TX(-\log D)$ and some of its properties.
Then, we show how $\tn{D}^{\log}_u\dbar$ can be used to study the deformation-obstruction spaces of  $\cM_{g,\mfs}(X,D,A)$ at $f$.
 
\begin{remark}
In the symplectic category, moduli spaces of relative maps were first studied by Ionel-Parker and Li-Ruan in \cite{IP1} and \cite{LR}. In \cite{FZ1}, we give a detailed account of their approaches and streamline some aspects of their construction. The work of Ionel-Parker does not include a dedicated Fredholm setup for relative maps. Deformation spaces are defined as a subspace of the classical deformation space and the obstruction spaces are not precisely identified. The approach of Li-Ruan is based on working with manifolds with cylindrical ends as in Symplectic Field Theory. In this approach, $\tn{D}^{\log}_u\dbar$ is defined over a subspace of 
$$
\Gamma\big(\Si^*,\big(u|_{\Si^*}\big)^*T(X-D)\big), \qquad \Si^*=\Si-u^{-1}(D),
$$ 
that is completed with weighted norms on the cylindrical parts of $\Si^*$ and $X-D$. These weighted norms control the behavior of $f$ at the infinity. As the following concrete definitions and structures indicate, our logarithmic Fredholm setup has the advantage of working with closed surfaces and is more user-friendly for applications such as explicitely identifying the deformation/obstruction spaces and localization calculations. 
\end{remark}

\noindent
Assume $(X,\om)$ is a symplectic manifold and $D\subset X$ is a smooth symplectic hypersurafce. Our construction of the complex vector bundle $TX(-\log D)$ (as well as an almost complex structure $J$ compatible with $D$) depends on 
 \begin{itemize}
 \item a Hermitian structure $(\rho,\nabla,\mfi)$ on $\cN_XD$, and
 \item a symplectic identification 
\bEq{Psi_e}
\Psi\colon \cN' \stackrel{\cong}{\lra} U\subset X
\eEq
of a neighborhood $U$ of $D$ in $X$ with a neighborhood $\cN'\subset \cN_XD$ of the zero section in the normal bundle 
$$
\pi\colon \cN_XD\equiv \frac{TX|_{D}}{TD}\lra D.
$$ 
\end{itemize}
Here, $\mfi$ is a complex structure compatible with the symplectic structure on the rank $2$ real normal bundle
$$
\cN_XD\cong TD^\om=\big\{v\in T_xX \colon x\in D,~\om(v,w)=0\quad \forall~w\in T_xD\},
$$
$\rho$ is a Hermitian metric on $(\cN_XD,\mfi)$, $\nabla$ is a $(\rho,\mfi)$-Hermitian connection, and the standard symplectic structure on $\cN'$ is determined by $\om|_D$ and $(\rho,\nabla)$ as in \cite[(3.1)]{FZ1}.
The connection $\nabla$ also defines a $\C$-linear injective homomorphism
$$
h_{\nabla}\colon \pi^* TD \lra  T\cN_XD
$$
that lifts a tangent vector $w\!\in\!T_xD$ to a ``horizontal" tangent vector $h_\nabla(w)\!\in\! T_v\cN_XD$, for all $v\!\in\! \cN_XD|_{x}$. The homomorphism $h_{\nabla}$ also gives rise  to a similarly denoted isomorphism
 \bEq{Tidentification_e}
 h_{\nabla}\colon \pi^* (TD\oplus \cN_XD) \stackrel{\cong}{\lra}  T\cN_XD, \qquad (v; w\oplus v') \to (v; h_\nabla(w)+v');
 \eEq
 see \cite[(2.6)]{FMZ2}. Let $\cR$ denote the tuple $(\Psi,(\rho,\nabla,\mfi))$ used above which is called a \textbf{regularization} in \cite{FMZ1}. Via (\ref{Tidentification_e}), an almost complex structure $J_D$ on $TD$ and the complex structure $\mfi$ on $\cN_XD$ define an almost complex structure $J'$ on $T\cN_XD$. Then, the identification $\Psi$  in (\ref{Psi_e}) can be used to extend $\Psi_*\big(J'|_{\cN'}\big)$ to an $\om$-tame almost complex structure $J$ over the entire~$X$. In \cite{FMZ1,FT1}, the space of such $D$-compatible almost K\"ahler structures $(\om,J)$ on $X$ is denoted by $\tn{AK}(X,D)$. The relative moduli spaces are defined for a larger class of $D$-compatible almost complex structures introduced in \cite[Sec.~3]{IP1}; also see \cite[(2.11)]{FZ2}. \\
 
 \noindent
The logarithmic tangent bundle arising from a regularization $\cR$ is defined to be
\begin{equation*}\begin{split}
&T_{\cR}X(-\log D)=\Big(\!\!
\big(\Psi_*\pi^*TD\!\oplus\!U\!\times\!\C\big)  
\!\sqcup\!T(X\!-\!D)\!\!\Big)\!\Big/\!\!\!\sim\,\lra U\!\cup\!(X\!-\!D)\!=\!X,\\
&~(\Psi_*\pi^*TD)\!\oplus\!U\!\times\!\C
\ni\big(\Psi(v); w\big)\!\oplus\!\big(\Psi(v);c\big)
\sim\nd_v\Psi\big(h_\nabla(w)\!+\!cv)\in T_{\Psi(v)}(X\!-\!D).
\end{split}\end{equation*}
In particular, 
 \bEq{logTdec_e}
  \Psi^*T_{\cR}X(-\log D)|_{U}= \pi^*TD|_{\cN'} \oplus \cN' \times \C.
 \eEq
Furthermore, the map $\iota_{\cR} \colon T_{\cR}X(-\log D)\lra TX$ defined by
\bEq{iota_e}
\iota_\cR(\ze)=\begin{cases}
\nd_v\Psi\big(h_\nabla(w)\!+\!cv),&\hbox{if}~
\ze\!=\!\big[(\Psi(v);w)\!\oplus\!(\Psi(v);c)\big],~v\!\in\!\cN';\\
\ze,&\hbox{if}~\ze\!\in\!T(X\!-\!V);
\end{cases}
\eEq
is a complex linear homomorphism whose restriction to $X\!-\!D$ is a bundle isomorphism. \\

\noindent
In \cite{FMZ1} and \cite{FMZ3}, we introduced topological notions of normal crossings symplectic divisor (and variety). We proved that, at the cost of deforming $\om$, every normal crossings symplectic divisor $D\subset X$ admits a compatible system $\cR$ of local linearizations generalizing the auxiliary data $\big(\Psi,(\rho,\nabla,\mfi)\big)$ used above in the smooth case; see \cite[Thm.~2.13]{FMZ1} and \cite[Thm.~3.4]{FMZ3}. 
We call such a compatible system $\cR$ of local linearizations a \textbf{regularization}.  In \cite{FMZ4}, we used a regularization $\cR$ to construct $T_{\cR}X(-\log D)$ and prove the following and a few other results for an arbitrary normal crossings symplectic divisor.

\begin{theorem}(\cite[Thm.~1.2]{FMZ4})\label{TLog_thm}
Suppose $(X,\om)$ is a symplectic manifold  and $D\!\subset\!X$ is a normal crossings symplectic divisor.  
\begin{enumerate}[label=(\arabic*),leftmargin=*]

\item\label{loddfn_it} 
An $\om$-regularization~$\cR$ for $D\!\subset\!X$ determines a vector bundle~$T_{\cR}X(-\log D)$ 
over~$X$ with a smooth vector bundle homomorphism 
\bEq{Ninc_e}
\iota_{\cR}\!: T_{\cR}X(-\log D)\lra TX.
\eEq
\item\label{logJ_it}
An $\cR$-compatible almost complex structure~$J$ on~$X$ determines a complex structure~$\mfi_{\cR,J}$ 
on the vector bundle $T_{\cR}X(-\log D)$ so that the bundle homomorphism~(\ref{Ninc_e}) is $\C$-linear.
\item\label{loginv_it}
The deformation equivalence class~$TX(-\log D)$ of~$T_{\cR}X(-\log D)$ 
depends only on the deformation equivalence class of $((X,\om),V)$.
\item\label{logsplit_it} If $D'\!\subset\!X$ is a smooth submanifold so that 
$D\!\cup\!D'\!\subset\!X$ is also an NC symplectic divisor and 
$D\!\cap\!D'$ contains no open subspace of~$D$, then
\bEq{logsplit_e} 
TX(-\log(D\!\cup\!D'))\!\oplus\! \cO_X(D')
\cong TX(-\log D)\!\oplus\!(X\!\times\!\C).
\eEq
\item\label{cTXV_it} We have 
\bEq{cTXV_e} 
c\big(TX(-\!\log D)\big)=
\frac{c(TX)}{1\!+\!\tn{PD}_X([D^{(1)}]_X)\!+\!\tn{PD}_X([D^{(2)}]_X)
\!+\!\ldots}\in H^*(X;\Q),
\eEq
where $D^{(k)}$ is the $k$-fold intersection locus of $D$.
The above equality holds in $H^*(X;\Z)$ if $D\!\subset\!X$ is a simple normal crossings divisor.
\end{enumerate}
\end{theorem}
\vskip.1in

\noindent
In what follows, in the light of Theorem~\ref{TLog_thm}.\ref{loginv_it}, we drop $\cR$ from the notation and simply denote $T_{\cR}X(-\log D)$ by $TX(-\log D)$ and  $\iota_{\cR}$ in (\ref{iota_e}) by $\iota$.\\

\noindent
Going back to the main discussion, the logarithmic linearization map (\ref{Dlog_e}) constructed in \cite{FT1} is a natural lift of (\ref{Ddbar_e}) that makes the following diagram commute:
\bEq{Ddbarnu_e}
\xymatrix{
& \Gamma(\Si,u^*TX(-\log D)) \ar[rrr]^{\tn{D}^{\tn{log}}_u \dbar} \ar[d]^{\iota}&&& \Gamma(\Si,\Om^{0,1}_{\Si}\otimes_\C u^*TX(-\log D))\ar[d]^{\iota}\\
& \Gamma(\Si,u^*TX)  \ar[rrr]^{\tn{D}_u\dbar} &&& \Gamma(\Si,\Om^{0,1}_{\Si}\otimes_\C u^*TX)\;.
}
\eEq

\noindent
With $\tn{D}_u^{\log} \dbar$ as in (\ref{Dlog_e}), similarly to the classical case, if $\tn{coker}(\tn{D}^{\log}_u\dbar)\!=\! 0$, by Riemann-Roch, the set of relative $J$-holomorphic maps of any fixed contact type $\mfs$ (over the fixed marked domain $C$) close to $f$ in (\ref{GE_e})) form an oriented smooth manifold of real dimension
$$
2\big( \tn{deg} (u^*TX(-\log D))\!+\! \dim_\C\! X (1\!-\!g)\big).
$$
Considering the deformations of the marked domain $C$ of the relative map $f$, we get the dimension formula (\ref{dlog_e})
and the deformation-obstruction long exact sequence 
\bEq{equ:long-def}
0\lra \tn{aut}(C) \stackrel{\de}{\lra} 
\Def_{\log}(u)\lra \Def_{\log}(f) \lra \Def(C) \stackrel{\de}{\lra} 
\Ob_{\log}(u)  \lra  \Ob_{\log}(f) \lra 0,
\eEq
such that
$$
\Def_{\log}(u)\!=\!\ker(\tn{D}^{\log}_u\dbar), \qquad \Ob_{\log}(u)\!=\!\coker(\tn{D}^{\log}_u\dbar),
$$
and $ \tn{aut}(C)$ is the Lie algebra of the automorphism group of $C$.
If $\tn{Ob}_{\log}(f)\!=\!0$, then a small neighborhood of $f$ in $\cM_{g,\mfs}(X,D,A)$ is a smooth orbifold of the expected dimension (\ref{dlog_e}).\\

\noindent
The long-exact sequence (\ref{equ:long-def}) is the long-exact sequence associated to a short-exact sequence of fine sheaves
$$
0\lra \cO(T\Si(-\log z))\lra \cO(u^*TX(-\log D))\lra \cO(\wt{\cN}_f)\lra 0
$$ 
over $\Si$, defined in the following way. \\

\noindent
In \cite[(4.18)]{FT2}, we show that the standard complex linear derivative map $\nd u\colon T\Si\!\lra\!TX$ gives rise to a \textbf{logarithmic derivative} map 
$$
\nd^{\log} u\colon T\Si(-\log \mf{z})\lra u^*TX(-\log D), \quad \tn{with}~\mf{z}=u^{-1}(D)\subset \Si,
$$ 
such that the following diagram commutes:
 \bEq{dulog_e1}
\xymatrix{
& T\Si(-\log \mf{z}) \ar[rr]^{\nd^{\log} u}\ar[d]^{\iota} && u^*TX(-\log D)\ar[d]^{\iota}\\
& T\Si \ar[rr]^{\nd u} && u^*TX\;.
}
\eEq
Here, the vertical maps $\iota$ are the natural $\C$-linear homomorphisms in (\ref{iota_e}) from the log tangent bundles $ T\Si(-\log \mf{z})$ and   $TX(-\log D)$ into the classical tangent bundles $T\Si$ and $TX$, respectively; also see \cite[Thm.~1.2]{FMZ2}.
By composing $\nd^{\log} u$ with the homomorphism $ T\Si(-\log z)\lra T\Si(-\log \mf{z})$, we can also define 
$$
\nd^{\log} u\colon T\Si(-\log z)\lra u^*TX(-\log D),
$$ 
where 
$$
z\equiv z_1+\cdots+z_k
$$ 
is the full divisor corresponding to all marked points. However, in what follows, we  will have $z=\mf{z}$; therefore, the two are the same.\\
\noindent 

\noindent
Briefly, the log derivative map $\nd^{\log} u$ is defined in the following way. Away from the contact points $\mf{z}=u^{-1}(D)\subset \{z_1,\ldots,z_k\}$, by the identification 
 $$
 TX(-\log D)|_{X-D}\stackrel{\iota}{\cong} TX|_{X-D},
 $$
 we have $\nd^{\log} u\cong  \nd u $.  For each marked point $z_a\!\in\!\mf{z}$, there are 
 \bIt
 \item a local coordinate $x$ on a sufficiently small neighborhood $\De\ni z_a$ of $z_a\in \Si$ (i.e., $z_a$ is $(x=0)$), 
 \item  and, a local chart $U_a$ around $u(z_a)\in D\subset X$, 
 \eIt
 such that 
 \bEn
 \item $U_a$ is identified with a neighborhood of $u(z_a)$ in the complex normal bundle $\pi\colon (\cN_XD,\mfi) \lra D$ as in (\ref{Psi_e});
 \item $T \cN_XD$ admits a decomposition $T \cN_XD\cong \pi^*TD \oplus \pi^* \cN_XD$ as in (\ref{Tidentification_e});
 \item with respect to the identifications in (1) and (2), $J|_{TU_a}$ coincides with $\pi^*\big( J_D\oplus \mfi)|_{U_a}$;
 \item the $J$-holomorphic map $u$ decomposes into horizontal and vertical components as 
 $$
 u(x)\!=\!(\ov{u}_a(x); \ze_a(x)) \in \cN_XD
 $$ 
meaning that $\ov{u}_a\colon \De\lra D $ is a $J_D$-holomorphic map into $D$ and $\ze_a$ is a section of $\ov{u}_a^* \cN_XD$;
 \item and, the section $\ze_a$ is holomorphic with respect to a $\C$-linear $\dbar$-operator $\ov{u}_a^*\tn{D}^{\cN}\dbar$ on $\ov{u}_a^*\cN_XD$ induced by $J$ along $D$.
 \eEn 
 Then, 
 \bIt 
 \item the holomorphic section $\ze_a$ decomposes as $\ze_a(x)=x^{s_a}\eta_a(x)$ with $\eta_a(0)\neq 0$,
 \item and, via the decomposition (\ref{logTdec_e}),  $\nd^{\log} u|_{\De}$ has an equation of the form
 $$
\nd^{\log} u= \nd \ov{u}_a\oplus \big(s_{a} \frac{\nd x}{x}+ \tn{holomorphic terms} \big),
 $$
mapping the local generating section $x\partial x$ of $T\Si(-\log \mf{z})|_{\De}$ to 
$$
\partial \ov{u}_a(x\partial x)\oplus \big(s_{a} + O(x)\big) \in T_{u(z_a)}D \oplus \C \cong TX(-\log D)|_{u(z_a)}
$$
\eIt
Note that at $x=0$, we have $(s_{a} + O(x))|_{x=0}=s_a$ and 
 \bEq{Logdu_e}
\nd^{\log}_{z_a} u (x\partial x)= 0 \oplus s_a\in T_{u(z_a)}D \oplus \C.
\eEq

\bDf{log-immersion_dfn}
We say $f\in \cM_{g,\mfs}(X,D,A)$ is a \textbf{log immersion} if  $u$ is an immersion away from $\mf{z}$ and $\mf{z}=z$ (i.e., all of the marked points are contact points with $D$).
\eDf

\noindent
By (\ref{Logdu_e}) and similarly to the classical case \cite[p. 284-285]{ST}, if $f=[u,\Si,\vec{z}=(z_1,\ldots,z_k)]$ is a log immersion, then $\nd^{\log} u$ is an embedding of vector bundles, the quotient 
\bEq{LogNorBundle_e}
\cN_f\equiv u^*TX(-\log D)/\big(\nd^{\log}_u~T\Si(-\log z)\big)
\eEq
is a complex vector bundle, and $\tn{D}^{\log}_u\dbar$ descends to an $\R$-linear Fredholm operator $\tn{D}^{\log}_{\cN_f}\dbar$ on smooth sections of $\cN_f$ such that 
\bEq{Ob_f}
\Def_{\log}(f)=\tn{ker}(\tn{D}^{\log}_{\cN_f}\dbar) \quad \tn{and}\quad \Ob_{\log}(f)=\tn{coker}(\tn{D}^{\log}_{\cN_f}\dbar).
\eEq
Here, the key point is that, by (\ref{Logdu_e}), $\nd^{\log} u$ is automatically injective at every contact point with $D$, whereas $\nd_{z_a} u$ could be zero if $s_a>1$. We call $\cN_f$ the \textbf{logarithmic normal bundle} of $f$. If only the relative marked points are concerned, one can define the logarithmic normal bundle of $u$ to be 
$$
u^*TX(-\log D)/\big(\nd^{\log}_u~T\Si(-\log \mf{z})\big).
$$
In our desired application, $\mf{z}=z$ and thus the two are the same.\\

\noindent
As usual, the Fredholm operator $\tn{D}^{\log}_{\cN_f}\dbar$ decomposes as 
\bEq{D=dbar+R_e}
\tn{D}^{\log}_{\cN_f}\dbar = \dbar_{\cN_f} + R
\eEq
such that $ \dbar_{\cN_f}$ is a $\C$-linear $\dbar$-operator defining a holomorphic structure on the log normal bundle $\cN_f$ and $R$ is an anti $\C$-linear zero-order operator depending on the Nijenhueis tensor of $J$. If $\nd^{\log} u$ is not an embedding or $\mf{z}\neq z$, we still obtain a short exact sequence of sheaves of $\cO_\Si$-modules
$$
0\lra \cO(T\Si(-\log z))\stackrel{\nd^{\log}u}{\lra} \cO(u^*TX(-\log D))\lra \cO(\wt\cN_f)\lra 0
$$
such that 
$$
\wt\cN_f= \cO(\cN_f)\oplus \cN^{\tn{tor}}_f
$$ 
is the direct sum of the sheaf of the holomorphic sections of an $(\dim_\C X-1)$-dimensional holomorphic vector bundle $\cN_f$ and a skyscraper sheaf $\cN^{\tn{tor}}_f$. Furthermore, $\tn{D}^{\log}_u\dbar$ descends to a Fredholm operator $\tn{D}^{\log}_{\cN_f}\dbar$ on sections of $\cN_f$ such that 
\bEq{Ob_f2}
\Def_{\log}(f)=\tn{ker}(\tn{D}^{\log}_{\cN_f}\dbar)\oplus H^0(\cN^{\tn{tor}}_f) \quad \tn{and}\quad \Ob_{\log}(f)=\tn{coker}(\tn{D}^{\log}_{\cN_f}\dbar);
\eEq
see \cite[Sec.~1.4]{Sh} or \cite{ST} for a general account of the discussion above.

\bLm{LItoSI_lmm}
Suppose $(X,D)$ is a symplectic log Calabi-Yau fourfold, $(\om,J)\in \AK(X,D)$, and 
$$
f\!=\![u,\P^1,z_1]\!\in\! \cM_{0,(s_A)}(X,D,A).
$$
Then, $\cN_f\cong \cO_{\P^1}(-(1+n))$, 
\bEq{Ob-n_e}
\Def_{\log}(f)=H^0(\cN^{\tn{tor}}_f)\quad\tn{and}\quad \dim_\R \Ob_{\log}(f)= 2n 
\eEq
where 
$$
n=\dim_\C H^0(\cN^{\tn{tor}}_f) \geq 0.
$$
In particular, 
$$
\cN_f\cong \cO_{\P^1}(-1)\quad \Leftrightarrow\quad \Ob_{\log}(f)=0 \quad\Leftrightarrow \quad  f \tn{ is a log immersion},
$$
and any of these implies that $u$ is somewhere injective and $\Def_{\log}(f)\!=\!0$.
\eLm

\begin{proof}
First, by (\ref{dlog_e}), we have 
$$
\dim \Def_{\log}(f) - \dim \Ob_{\log}(f)=
2\Big(0+(1-0)(2 -3)+1\Big)=0
$$
By Riemann-Roch, 
$$
\dim_\R \tn{ker}(\tn{D}^{\log}_{\cN_f}\dbar)- \dim_\R \tn{coker}(\tn{D}^{\log}_{\cN_f}\dbar)=
2\big( \tn{deg}(\cN_f) \!+1\big).
$$
From (\ref{Ob_f2}), we conclude that $\cN_f\cong \cO_{\P^1}(-(1+n))$.
Since $-(1+n)<0$, by \cite[Thm.~1']{HLS} , for any choice of a compact operator $R$ in (\ref{D=dbar+R_e}), we have  $\tn{ker}~\tn{D}^{\log}_{\cN_f}\dbar=0$, which gives us (\ref{Ob-n_e}).
The conditions $\cN_f\cong \cO_{\P^1}(-1)$, $\Ob_{\log}(f)=0$, and $f$ being a log immersion are all equivalent to $n\!=\!0$.
Every holomorphic map $h\colon \P^1\to \P^1$ of degree $2$ or higher has at least two branch points. Therefore, since $f$ has only one contact marked point $z_1$, if $f$ is a log immersion then $u$ must be somewhere injective and an immersion away from $z_1$; i.e., $n=0$. \end{proof}

\begin{remark}
In the context of Lemma~\ref{LItoSI_lmm}, if $u$ is somewhere injective with cusp points 
$$
w_1,\ldots,w_k\in u^{-1}(X-D)
$$ 
of orders $b_1,\ldots,b_k$ (see \cite[Sec~1.5]{Sh}), then $\cN^{\tn{tor}}_f$ is the direct sum sky-scraper sheaf $\oplus_{i=1}^k \C_{w_i}^{b_i}$; in particular, $n=\sum_{i=1}^k b_i$. The marked point $z_1=\infty$ can also be a cusp point of $u$, however, since the second term on the righthand side of (\ref{Logdu_e}) is non-zero at $z_1$, from the logarithmic perspective, $d_{z_1}^{\log} u$ is non-zero and $z_1$ behaves like a smooth point. 
\end{remark}

\begin{remark}
The discussion above only concerns maps in the virtually main stratum $\cM_{g,\mfs}(X,D,A)$ of the compactified moduli space $\ov\cM_{g,\mfs}(X,D,A)$. Together with some ad-hoc dimension counting, this is sufficient for proving Theorem~\ref{Compactification_th}. As noted in \cite{GV}, in order describe the deformation-obstruction spaces along maps in other strata of $\ov\cM_{g,\mfs}^{\tn{rel}}(X,D,A)$, one must systematically replace $TX$ by $TX (-\log D)$. We illustrate how this should be done in the analysis of the virtual normal bundle to the fixed loci components in the localization calculations of Section~\ref{loc_sec}.
\end{remark}

\newtheorem*{proofof-Compactification_th}{Proof of Theorem~\ref{Compactification_th}}
\begin{proofof-Compactification_th}

In \cite[Prp.~1.7 and Crl.~1.9]{FT2}, we prove that if a pair $(X,D)$ is positive in the sense of \cite[Dfn.~4.7]{FZ1} or equally \cite[Dfn.~1.6]{FT2}, and $E>0$ is an arbitrary large number, then there exists a Baire subset  
$$
\tn{AK}^{\tn{reg}}(X,D)\!\in\!\tn{AK}(X,D)
$$ 
of the second category such that for every $(\om,J)$ in this set and $A\in H_2(X,\Z)$ with $\om(A)\leq E$,
\bEn
\item the moduli space $\cM^{\star}_{0,(s_A)}(X,D,A)$ is cut transversely and is a smooth manifold of real dimension 
\bEq{LD_e}
2\big(\ll c_1(TX(-\log D)),A\rr+(\dim_\C X-3)+1\big);
\eEq
\item the image of  $\ov\cM^{\log}_{0,(s_A)}(X,D,A)-\cM^{\star}_{0,(s_A)}(X,D,A)$ under the evaluation map
\bEq{Lev_e}
\ev\colon\!\cM^\star_{0,(s_A)}(X,D,A)\!\lra D
\eEq
lies in the image of smooth maps from finitely many smooth even-dimensional manifolds of at least $2$ real dimension less than (\ref{LD_e}),
\item and consequently, the map (\ref{Lev_e}) defines a pseudo-cycle of real dimension (\ref{LD_e}) in $D$.
\eEn
The same holds for the relative compactification $\ov\cM^{\tn{rel}}_{0,(s_A)}(X,D,A)$ instead of $\ov\cM^{\log}_{0,(s_A)}(X,D,A)$; see \cite{IP1, FZ1, FZ2}. Similarly to the proof of the classical results \cite[Thm.~6.6.1]{MS} and \cite[Prp.~3.21]{RT}, the proof of (1)-(3) above is by showing that for generic $J$, 
\bIt
\item each stratum of the simple part $\ov\cM^{\star}_{0,(s_A)}(X,D,A)$ is cut transversely, 
\item and, under the positivity condition, non-simple or multiple-cover maps (more precisely, their image under $\tn{ev}$) happen in real codimension $2$ or higher.
\eIt
\noindent
The proof of (2) involves replacing a non-simple map $f$ with an underlying simple map $f'$ with multi-nodes (i.e., points at which more than one component intersect). This is done by (i) collapsing the ghost bubbles
(ii) replacing each multiple-cover bubble component by its image curve (or the underlying simple map), and (iii) collapsing each sub-tree of the bubbles whose components have the same image.\\

\noindent 
Under the weaker semi-positivity assumption, instead, the codimension of the collapsed stratum is a complicated formula that depends on the number of multi-nodes and a few other factors; see \cite[(3.42)]{RT} and \cite[(4.79)]{FT2}. In particular, the real codimension will be a positive even number if 
$f'$ has more than one irreducible component.  Therefore, if $(X,D)$ is semi-positive, as is the case in Theorem~\ref{Compactification_th}, the same argument as in the positive case shows that 
\bEn
\item each stratum of $\ov\cM^{\star}_{0,(s_A)}(X,D,A)$ is cut transversely, 
\item and, non-simple or multiple-cover maps $f$ for which $f'$ has at least two components  happen in real codimension $2$ or higher.
\eEn
Since, by assumption, the expected dimension (\ref{LD_e}) is zero in Theorem~\ref{Compactification_th}, we conclude that for generic $J$
\bIt
\item the moduli space $\cM^{\star}_{0,(s_A)}(X,D,A)$ is cut transversely and is a discrete set of points;
\item every stratum in $\ov\cM^{\star}_{0,(s_A)}(X,D,A)-\cM^{\star}_{0,(s_A)}(X,D,A)$ has negative dimension (because we get $2$ real codimension for each node), thus is empty;
\item and, the only strata of multi-cover maps in 
$$
\cM^{\tn{mc}}_{0,(s_A)}(X,D,A)\equiv \ov\cM_{0,(s_A)}(X,D,A)-\ov\cM^{\star}_{0,(s_A)}(X,D,A)
$$
with non-negative dimension are those which are multiple covers of a single somewhere injective curve; i.e.,
$$
\cM^{\tn{mc}}_{0,(s_A)}(X,D,A)\cong \coprod_{\substack{B\in H_2(X,\Z)\\ dB=A, d>1}}\ov\cM_{0,(d)}\big(\P^1,\infty,[d \P^1]\big)^*\times \cM^\star_{0,(s_B)}(X,D,B)
$$
\eIt
Here, $\ov\cM_{0,(d)}\big(\P^1,\infty,[d \P^1]\big)^*$ are the relative moduli spaces in (\ref{s_B-rmk}).\\

\noindent
In order to finish the proof of Theorem~\ref{Compactification_th}, it remains to show that 
\bIt
\item $\cM^{\star}_{0,(s_A)}(X,D,A)$ consists of finitely many points with no accumulation point in $\cM^{\tn{mc}}_{0,(s_A)}(X,D,A)$,
\item and, the subset of such regular almost K\"ahler structures $\tn{AK}^{\tn{reg}}(X,D)\!\subset \!\!\tn{AK}(X,D)$ is connected in each deformation equivalence class of $\om\in \Symp_{\log}(X,D)$.
\eIt
\end{proofof-Compactification_th}

\noindent
The first statement follows from the super-rigidity/automatic transversailty Lemma~\ref{LItoSI_lmm} in the following way. 
Suppose $(X,D)$ is a symplectic log Calabi-Yau fourfold,  
$$
f=[u,\Si,z_1] \in \ov\cM_{0,(d)}\big(\P^1,\infty,[d \P^1]\big)^*\times \cM^\star_{0,(s_B)}(X,D,B)\subset \cM^{\tn{mc}}_{0,(s_A)}(X,D,A)
$$
and $\{f_i\}_{i\in \N}$ is a sequence in $\cM^\star_{0,(s_A)}(X,D,A)$ that converges to $f$. 
Let 
$$
\ov{f}\equiv [\ov{u},\P^1,\infty]\in \cM^\star_{0,(s_B)}(X,D,B)
$$ 
denote the simple map underlying $f$. By Lemma~\ref{LItoSI_lmm} and since $\ov{f}$ is cut transversely for all $B$ with $\om(B)<\om(A)$, we have 
$$
\ov{u}^* TX(-\log D)\cong  T\P^1(-\log \infty) \oplus \cN_{\ov{f}}\cong \cO(1)\oplus \cO(-1).
$$  
Therefore,  
$$
u^* TX(-\log D)\cong  \cO(d)\oplus \cO(-d),
$$
and the holomorphic sections of $ \cO(d)$ correspond to vector fields tangent to the image of $u$. 
Consequently, as in \cite[Prp.~B.1]{W2},  there is $i_0$ such that for all $i>i_0$, $f_i$ has the same image as $f$. That contradicts the simpleness of $f_i$ if $d\!>\!1$.  This finishes the proof of the first bullet above. \\

\noindent
Suppose $(\om_1,J_1)$ and $(\om_2,J_2)$ are two regular almost K\"ahler structures on $(X,D)$ such that $\om_1$ and $\om_2$ are deformation equivalent in $\tn{Symp}_{\log}(X,D)$. By considering the space of all paths $(\om_t,J_t)_{t\in[0,1]} $ connecting $(\om_1,J_1)$ and $(\om_2,J_2)$ and the family moduli space 
$$
\ov\cM_{0,(s_A)}\big(X,D,A; \{J_t\}_{t\in [0,1]}\big)\lra [0,1],
$$
the same proof as above shows that for a generic path $(\om_t,J_t)_{t\in[0,1]} $

\bEn
\item the moduli space $\cM^{\star}_{0,(s_A)}\big(X,D,A; \{J_t\}_{t\in [0,1]}\big)$ is an (oriented) smooth $1$-manifold;
\item every stratum in 
$$ 
\ov\cM^{\star}_{0,(s_A)}\big(X,D,A; \{J_t\}_{t\in [0,1]}\big)-\cM^{\star}_{0,(s_A)}\big(X,D,A; \{J_t\}_{t\in [0,1]}\big)
$$ 
has negative dimension (because we get $2$ real codimension for each node) and therefore is empty;
\item and, the only strata of multi-cover maps in 
$
\cM^{\tn{mc}}_{0,(s_A)}\big(X,D,A;\{J_t\}_{t\in [0,1]}\big)
$
with non-negative dimension are those which are multiple covers of a single somewhere injective curve.
Suppose 
$$
f=[u,\Si,z_1] \in \ov\cM_{0,(d)}\big(\P^1,\infty,[d \P^1]\big)^*\times \cM^\star_{0,(s_B)}(X,D,B; J_t)\subset \cM^{\tn{mc}}_{0,(s_A)}(X,D,A; J_t), 
$$
for some $t\!\in\! [0,1]$ and $\{f_i\}_{i\in \N}$ is a sequence in $\cM^\star_{0,(s_A)}\big(X,D,A; \{J_t\}_{t\in [0,1]}\big)$ that converges to~$f$. We may assume that the underlying simple curve $\ov{f}=[\ov{u},\P^1,\infty]\in \cM^\star_{0,(s_B)}(X,D,B; J_t)$ is not a regular point of the projection
\bEq{tProj_e}
\cM_{0,(s_B)}\big(X,D,B; \{J_t\}_{t\in [0,1]}\big)\lra [0,1],
\eEq
otherwise, the same reasoning as above shows that the only other curves in 
$$
\ov\cM_{0,(s_A)}(X,D,A; \{J_t\}_{t\in [0,1]})
$$ 
near $f$ are of the form $\ov{u'}\circ h'$ where $[\ov{u'},\P^1,\infty]\in \cM^\star_{0,(s_B)}(X,D,B;\{J_t\})$ is a deformation of $[\ov{u},\P^1,\infty]$ and $h'$ defines an element of $ \ov\cM_{0,(d)}\big(\P^1,\infty,[d \P^1]\big)^*$. Furthermore, for a generic path $(w_t,J_t)_{t\in [0,1]}$, every critical point of the projection map (\ref{tProj_e}) satisfies 
$$
\dim_\R \Ob_{\log}(\ov{f}) =1.
$$
By the second identity in (\ref{Ob-n_e}), the latter is impossible. This finishes the proof of (2) and thus Theorem~\ref{Compactification_th}.\qed
\eEn

\section{Relative localization and the proof of Theorem~\ref{GW_th}}\label{loc_sec}

In this section, we first derive Theorem~\ref{GW_th} from Theorem~\ref{Compactification_th} and then explain the localization calculation of (\ref{mc_e}).

\newtheorem*{proofof-GW_th}{Proof of Theorem~\ref{GW_th}}
\begin{proofof-GW_th} 
By Theorem~\ref{Compactification_th}, for generic choice of $J$,  the moduli space $\ov\cM_{0,(s_A)}(X,D,A)$ decomposes into a disjoint union of closed components
\bEq{Decomp_e}
\ov\cM_{0,(s_A)}(X,D,A)=\cM^\star_{0,(s_A)}(X,D,A)\cup \cM^{\tn{mc}}_{0,(s_A)}(X,D,A)
\eEq
where $\cM^\star_{0,(s_A)}(X,D,A)$ is a finite set of (oriented) points and $\cM^{\tn{nc}}_{0,(s_A)}(X,D,A)$ is a disjoint union of positive-dimensional closed oriented orbifolds. Furthermore, for two deformation equivalent regular almost complex structures $J_1$ and $J_2$, there is a path of almost complex structures $\{J_t\}_{t\in [0,1]}$ such that  $\ov\cM_{0,(s_A)}(X,D,A;\{J_t\}_{t\in [0,1]} )$ similarly decomposes into a disjoint union of closed components
$$
\ov\cM_{0,(s_A)}\big(X,D,A;\{J_t\}_{t\in [0,1]}\big)\cong \coprod_{\substack{B\in H_2(X,\Z)\\ dB=A}}\ov\cM_{0,(d)}\big(\P^1,\infty,[d \P^1]\big)^*\times \cM^\star_{0,(s_B)}\big(X,D,B;\{J_t\}_{t\in [0,1]}\big).
$$
In conclusion, the relative GW invariant 
$$
n_A = \#\cM^\star_{0,(s_A)}(X,D,A)\in \Z
$$
is well-defined and counts the number of logarithmically immersed degree $A$ rational curves in $X$ with maximal tangency at a single point with $D$.\\

\noindent
Since a generic path $\{J_t\}_{t\in [0,1]}$ preserves the decomposition (\ref{Decomp_e}), we may artificially define 
$$
\Q\ni N_A= \sum_{\substack{B\in H_2(X,\Z)\\ dB=A, d>0}}\tn{mc}(d,s_B)~n_{B}\qquad \forall~A\!\in\!H_2(X,\Z)
$$
as in \cite[Thm.~1.5]{BS}.
However, by considering the Euler class of the obstruction bundle over each closed component of $\ov\cM_{0,(s_A)}(X,D,A)$, we can equip $\ov\cM^{\tn{rel}}_{0,(s_A)}(X,D,A)$  with a natural virtual fundamental (or \textbf{VFC}) class such that 
$$
N_A=\int_{[\ov\cM^{\tn{rel}}_{0,(s_A)}(X,D,A)]^{\tn{VFC}}} 1.
$$ 
This is a simple case of a similar construction/definition of \textbf{VFC} in \cite[Thm.~1.2.(1)]{Z3}. \\

\noindent
It just remains to explain the calculation of the multiple-cover contributions (\ref{mc_e}) which will occupy the rest of this section.
\qed
\end{proofof-GW_th}

\begin{remark}
To be consistent with the setup in the items (1)-(3) including (\ref{Lev_e}), we may as well define 
$$
n_A = \tn{deg}\big(\tn{ev}\big(\cM^\star_{0,(s_A)}(X,D,A)\big)\in H_0(D,\Z)\cong \Z.
$$
\end{remark}
\vskip.2in 
\noindent
Consider the $S^1$-action 
$$
[x_0,x_1]\lra  [x_0,t x_1]\qquad \forall~t\in S^1\subset \C^*
$$
on $\P^1$ corresponding to weights $(\al_0,\al_1)=(0,-1)$ in \cite[Sec.~27.1]{Mir}. Let $y_0\!=\!x_1/x_0$ and $y_1=x_0/x_1$ be the affine coordinates around 
$$
p_0\equiv [1,0]\quad \tn{and}\quad p_1\equiv [0,1]\in \P^1.
$$ 
The weights of the action on $T_{p_0}\P^1$ (generated by $\partial y_0$) and $T_{p_1} \P^1$ (generated by $\partial y_1$)  are $1$ and $-1$, respectively. Different lifts of the given action to $\cO(1)\lra\P^1$ are classified by the action-weights $(m+1,m)$ on $\big(\cO(1)|_{p_0},\cO(1)|_{p_1}\big)$.
Since $T\P^1(-\log p_1)|_{p_0}$ and   $T\P^1(-\log p_1)|_{p_1}$ are generated by $\partial y_0$ and $y_1\partial y_1$, the line bundle 
$$
T\P^1(-\log p_1)\cong \cO(1)
$$ has weights $(1, 0)$ at $(p_0,p_1)$.\\

\noindent
Let 
\bEq{GC_e}
h\colon \P^1\lra \P^1,\qquad h(y_0)=y_0^d,
\eEq
be the Galois $d$-covering of $\P^1$ fully ramified at $p_0$ and $p_1$. By \cite[Exe.~27.2.3]{Mir}, for $m$ as above, the weights of the lifted action on 
$$
H^1(\P^1,h^*\cO(-1))= H^0(\P^1, \om_{\P^1}\otimes h^* O(1))^\vee
$$
are 
$$
\aligned
&-\Big(\frac{j}{d}+ m\Big),\qquad 1\leq j \leq d-1,  
\endaligned
$$
where the $-$ sign comes from the Serre duality\footnote{The statement of \cite[Ex.~27.2.6]{Mir} is missing the sign factor arising from the Serre duality.}. Therefore, the product of the weights of the $S^1$-action on $H^1(\P^1,h^*\cO(-1))$ is 
\bEq{O(-1)LC_e}
\frac{(-1)^{d-1}}{d^{d-1}} \prod_{j=1}^{d-1} \big(j+md\big).
\eEq
This number will determine the contribution of the obstruction bundle below with $O(-1)$ being the logarithmic normal bundle of a log immersed curve $\ov{u}$. \\

\noindent
The $S^1$-action on $\P^1$ naturally lifts to an $S^1$-action on $\ov\cM^{\tn{rel}}_{0,(d)}\big(\P^1,p_1=\infty,[d \P^1]\big)^*$. 
The fixed point loci of this action come in families of various dimension and are characterized by (unordered) partitions 
$$
\cP \colon \big(d=d_1+\cdots+d_k\big), \quad d_i>0,
$$   
of $d$ in the following way. 
\bIt
\item If $k=1$, the fixed point locus is a single point and corresponds to the degree $d$ Galois cover  in (\ref{GC_e}) of the fixed base simple curve $\ov{f}=[\ov{u},\P^1,p_1]\in \cM^\star_{0,(s_B)}(X,D,B)$.
\item If $k\geq 2$, the fixed locus $\ov\cM_\cP$ corresponding to $\cP$ is the compactification of the open dense subset $\cM_\cP\cong \cM_{0,k+1}/\sim$ described below. Figure~\ref{fixedloci_fig} illustrates the domain and image of an element in $\cM_{\cP}$ with $k=3$.
\eIt
The curves in $\cM_u$ consist of $k$ (unordered) Galois covers
$$
h_i\colon (\P^1,\infty) \lra (\P^1,p_1), \qquad i=1,\ldots,k,
$$ 
of degrees $d_1,\ldots,d_k$ of $\ov{f}$ connected to a genus $0$ $(k+1)$-marked curve that is mapped $(ds_B):1$ to the fiber 
$$
\P^1\cong F_q \subset \P_XD=\P(\cN_XD\oplus \C)
$$ 
over $q= \ov{u}(p_1)\in D$; see Appendix~\ref{RC_A} for a brief description of the relative moduli space and the notations used here. 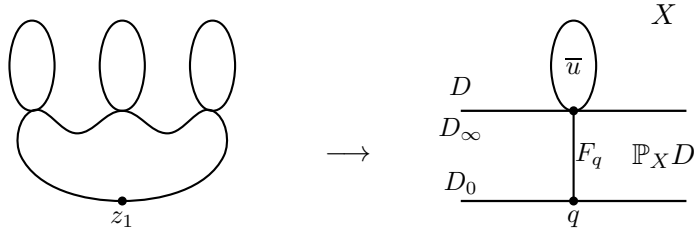
\begin{figure}
\begin{pspicture}(38,-1)(11,2)
\psset{unit=.3cm}
\psline[linewidth=.1](70,-2)(80,-2)
\psline[linewidth=.1](70,2)(80,2)

\rput(79,6.3){{$X$}}
\rput(79,0){$\P_XD$}

\rput(70,3.1){\small{$D$}}
\rput(70,1.2){\small{$D_{\infty}$}}
\rput(70,-1.2){\small{$D_{0}$}}
\pscircle*(75,-2){.2}\rput(75,-2.8){\small{$q$}}
\psccurve(74,4)(75,2)(76,4)(75,6)\rput(75,4){\small{$\ov{u}$}}
\pscircle*(75,2){.2}
\psline(75,2)(75,-2)\rput(75.7,0){\small{$F_q$}}

\psccurve(50,4)(51,2)(52,4)(51,6)
\psccurve(54,4)(55,2)(56,4)(55,6)
\psccurve(58,4)(59,2)(60,4)(59,6)
\psccurve(51,2)(53,1)(55,2)(57,1)(59,2)(59.5,0)(55,-2)(50.5,0)
\pscircle*(55,-2){.2}\rput(55,-2.8){\small{$z_1$}}
\rput(65,0){$\lra$}

\end{pspicture}
\caption{A relative map with $k\!=\!3$ into the expanded degeneration $X[1]= X \cup \P_XD$.}
\label{fixedloci_fig}
\end{figure}
In a suitable coordinate $z$ (so that the intersection with the divisor $D_0\subset \P_XD$ happens at $z=\infty$), the map to $F_q$ is given by the polynomial 
\bEq{ds_B}
\vr(z)=\Big(\prod_{i=1}^k (z-z_i)^{d_i})\Big)^{s_B}\colon \P^1\to F_q\cong \P^1
\eEq
and is fully ramified at the infinity. Since the nodes of $f$ are unordered,  the notation $\cM_{0,k+1}/\!\!\sim$ means $\cM_{0,k+1}$ divided by a permutation group of marked points $z_i$ and $z_j$ with $d_i=d_j$. The identification $\cM_\cP\cong \cM_{0,k+1}/\!\!\sim$ is given by 
$$
\vr(z) \Leftrightarrow  [\P^1,z_1,\ldots,z_k,\infty]\in \cM_{0,k+1}/\!\!\sim.
$$

\begin{remark}
We may alternatively put an order on the nodal points $z_1,\ldots,z_k$ and divide (\ref{Contmu_e}) by $k!$.
\end{remark}

\begin{remark}
If $s_B\!=\!1$, then $\ov\cM_{\cP}=\ov\cM_{0,(d)}^{\tn{rel}}\big(\P^1,p_1,[d \P^1]\big)$.
If $s_B\!>\!1$, due to the extra factor of $s_B$ in (\ref{ds_B}), the algebraic moduli structure or the orbifold structure of the moduli space of $d$-covers of $\ov{f}$ slightly differs from the standard structure  on $\ov\cM_{0,(d)}^{\tn{rel}}\big(\P^1,p_1,[d \P^1]\big)$. Following \cite{GPS}, the correct moduli structure is indicated by a superscript $*$.
\end{remark}

\begin{remark}
The log compactification constructed in \cite{FT1} is smaller than the relative compactification. In fact, in the log compactification $\ov\cM_{0,(d)}^{\tn{log}}\big(\P^1,p_1,[d \P^1]\big)^*$, the fixed loci of the $S^1$-action are still indexed by $\cP$, the open dense part $\cM_\cP$ of $\ov\cM_\cP^{\log}$  is the same as above, and $\ov\cM_\cP^{\log} \cong \ov\cM_{0,k+1}/\!\!\sim$. In the relative compactification, however, $\ov\cM_\cP$ is a complicated blowup of $\ov\cM_{0,k+1}/\!\!\sim$ away from $\cM_{0,k+1}/\!\!\sim$.
\end{remark}

\noindent
In the following, we first explain the localization contribution of the isolated point (\ref{GC_e}) when $k=1$, and then describe the contributions of other strata corresponding to $k>1$.\\

\noindent
For 
$$
f=[u=\ov{u}\circ h, C=(\P^1,\infty)]\in \cM_{0,(s_A)}(X,D,A),
$$
if $\ov{u}$ is a log immersion, we have 
$$
\ov{u}^*TX(-\log D)= T\P^1(-\log p_1) \oplus \cN_{\ov{f}} \cong \cO(1)\oplus \cO(-1)
$$
and
the Deformation-Obstruction long exact sequence
$$
0\lra \tn{aut}(C) \lra
\Def_{\log}(u)\lra \Def_{\log}(f) \lra \Def(C) \lra
\Ob_{\log}(u)  \lra  \Ob_{\log}(f) \lra 0,
$$ 
in  (\ref{equ:long-def}) reads 
$$
\aligned
&0\lra H^0\big(\P^1, T\P^1(-\log \infty)\big)\lra H^0\big(\P^1, h^*T\P^1(-\log p_1)\big)\lra \Def_{\log}(f) \lra 0,\\
&0  \lra H^1\big(\P^1, h^*\cN_{\ov{f}}\big) \lra  \Ob_{\log}(f) \lra 0.
\endaligned
$$
Since the line bundle  $T\P^1(-\log p_1)\cong \cO(1)$ has weights $(1, 0)$ at $(p_0,p_1)$, by \cite[Exe.~27.2.3]{Mir}, the localization contribution of $\Def_{\log}(f)$ is 
$$
\frac{d! }{d^{d-1}}.
$$
Therefore, by (\ref{O(-1)LC_e}) and considering the discrete automorphism factor $1/d$ of $h$ in (\ref{GC_e}), the contribution of the $d$-fold Galois cover to $N_A$ is 
\bEq{LCm_e}
\frac{(-1)^{d-1}}{d^2}\frac{ \prod_{j=1}^{d-1} \big(j+md\big)}{(d-1)!}=\frac{(-1)^{d-1}}{d^2}{d(m+1)-1 \choose d-1}=\frac{1}{d^2}{-md-1 \choose d-1}.
\eEq

\noindent
In the proof of \cite[Prp.~6.1]{GPS}, the authors state that the right value of $m$ to consider for the other contributions to vanish is 
$$
m=1-s_B,
$$ 
which yields the coefficient
\bEq{mcdssB_e}
\tn{mc}(d,s_B)= \frac{(-1)^{d-1}}{d^2}{d(2-s_B)-1 \choose d-1}=\frac{1}{d^2}{d(s_B-1)-1 \choose d-1}.
\eEq
The justification (for the vanishing of other terms) is left to the reader and the claim is stated to follow from a similar argument as in the proof of \cite[Thm. 5.1]{BP} and examining the obstruction space in \cite{GV}.

\begin{remark} The moduli space of concern is (more or less) a simple case $(r=1, \alpha_1=d, b_1=(d))$ of \cite[Dfn.~3.1]{BP}. When $s_B=1$, the rank $d-1$ obstruction bundle 
$$
\tn{Ob}(d,1)\lra  \ov\cM^{\tn{rel}}_{0,(d)}\big(\P^1,\infty,[d \P^1]\big)
$$
is $\cR^1\pi_* f^*\mathcal{O}(-1)$; see the proof of \cite[Lmm.~5.9]{B2}. We could not find any reference in which the obstruction bundle is explicitely identified as above when $s_B>1$. The proof of \cite[Thm. 5.1]{BP} also concerns the case $s_B=1$ and the chosen weights on the logarithmic normal bundle $\mathcal{O}(-1)$ are $(-1,0)$. Items (i)-(iv) in \cite[p.~387]{BP}  explain the reasons for the vanishing of the other strata in the fixed locus (i.e., $\ov\cM_\cP$ with $k>1$). 
\end{remark}

\begin{remark}
In \cite{GPS}, the base curve $\ov{u}\colon \P^1\lra X$ is (implicitly) assumed to be an immersion so that its (regular) normal bundle is defined and is the line bundle $\cO(s_B-2)\lra \P^1$. In the light of Lemma~\ref{LItoSI_lmm}, the map $\ov{u}$ can have a cusp point at the intersection point with $D$ and still be a smooth point of the moduli space $\cM_{0,(s_B)}(X,D,A)$. The local contribution should still be the same as (\ref{mcdssB_e}) because the cusp point can be smoothed out without affecting the transversality.
\end{remark}

\begin{remark}\label{GVsteup_rmk}
In the setup of Graber-Vakil's paper \cite{GV}, they require an $S^1$-action on the entire target $X$ that preserves $D$ (and the induced action on $\cN_XD$ has non-zero weights). In the approach of \cite{GPS} reviewed above, we only have an action on the domain $\P^1$ of $\ov{u}$. Then, we need to identify the obstruction bundle $\tn{Ob}(d,s_B)$ and lift the given $S^1$-action on the base to the fibers of $\tn{Ob}(d,s_B)$. This can be done in many ways and different choices correspond to different values of $m$ above. The considered weight on $\cN_XD|_{q}$ in \cite{GPS} is zero. Nevertheless, even though we don't have a global $S^1$-action on $X$, a neighborhood of $\ov{u}$ can be identified with a neighborhood of $\P^1$ in the normal bundle $\cO(s_B-2)$ and the action can be (infinitesimally) lifted to such a neighborhood. With respect to local coordinates $(y_1,c_1)$ on the total space of $\cO(s_B-2)$ at $p_1$, the intersection of $D$ with such a local neighborhood is given by the local model equation
$$
(c_1-y_1^{s_B}=0)\subset \C^2.
$$ 
In order for this equation (in other words, the intersection of $D$ with this neighborhood) to be invariant under the $S^1$-action, the weight of the action on $c_1$ must be $-s_B$. That means the weight of the action on $\cO(s_B-2)|_{p_1}$ must be $-s_B$ which corresponds to $m=1$ and 
\bEq{m=1GC_e}
\frac{(-1)^{d-1}}{d^2}{d(m+1)-1 \choose d-1}=\frac{(-1)^{d-1}}{d^2}{2d-1 \choose d-1}
\eEq
in (\ref{LCm_e}).
\end{remark}

\noindent
In the rest of this section, using the trick in Remark~\ref{GVsteup_rmk} and the resulting weight $-s_B$ on the normal bundle $\cN_XD$, we  explain the localization contributions of the fixed curves in $\cM_\cP$ with $k>1$. 

\begin{remark}
In \cite{GV}, the authors assume that the $S^1$-action fixes $D$ to obtain the relative virtual localization formula \cite[Thm.~3.6]{GV}. The local action described in Remark~\ref{GVsteup_rmk} only preserves $D$. 
Nevertheless, if $s_B\!>\!1$, the long-exact sequence in \cite[p.~14]{GV} splits into easy to understand terms from which we obtain an explicit relative virtual localization formula.
\end{remark}

\noindent
Since the case $s_B$ is well-studied, assume $s_B>1$. For 
$$
f=\big(u\colon C=(\Si,z_1)\lra (X[1],D_0)\big) \in \cM_\cP
$$ 
illustrated in Figure~\ref{fixedloci_fig}, let 
$$
f_1=\big(u_1\colon C_{1}=(\Si_1,(z_{1i})_{i=1}^k)\lra (X,D)\big)=\bigcup_{i=1}^k\big(u_{1i}=\ov{u}\circ h_{1i}\colon C_{1i}=(\P^1_{1i},\infty)\lra (X,D)\big)
$$ 
denote the union of components corresponding to the $k$ Galois covers and 
$$
f_2=\Big(u_2\colon C_2=\big(\P^1_2,(z_{2i})_{i=1}^k\cup z_1\big)\lra (\P_XD,D_\infty\cup D_0)\Big)
$$ 
denote the rubber component in $\P_XD$ given by (\ref{ds_B}). As noted in \cite{GV}, the analysis of the virtual normal bundle to such a component is identical to the case of ordinary stable maps with the bundle $TX$ systematically replaced by $TX (-\log D)$.
In fact, the contribution of $\ov\cM_\cP$ to $N_{0,dB}$ is of the form
\bEq{Contmu_e}
\tn{Cont}(\cM_\cP)=\frac{\prod_{i=1}^k (d_is_B)}{|\tn{Aut}(\cP)| \prod_{i=1}^k d_i}\frac{e(\tn{Ob}_{\log}(f)^{\tn{mov}})e(\tn{aut}(C)^{\tn{mov}})}{e(\tn{Def}_{\log}(f)^{\tn{mov}}) e(\tn{Def}(C)^{\tn{mov}})}
\eEq
where 
\bIt
\item $\tn{Def}_{\log}(f)$ and $\tn{Ob}_{\log}(f)$ are calculated via the long-exact sequence
\bEq{DOloc_e}
\aligned
0 &\lra \tn{Def}_{\log}(f)
 \lra H^0(\Si_1,u_1^*TX(-\log D))\oplus  \wt{H}^0(\P^1_2,u_2^*
T\P_XD(-\log D_\infty \cup D_0))
\stackrel{\tn{Res}}{\lra} \bigoplus_{i=1}^k T_{q}D\\
&\lra \tn{Ob}_{\log}(f) \lra H^1(\Si_1,u_1^*TX(-\log D))\oplus  H^1(\P^1_2,u_2^*
T\P_XD(-\log D_\infty \cup D_0))\lra 0
\endaligned
\eEq
as in \cite[p.~14]{GV};
\item $\tn{Aut}(\cP)$ is the symmetry group of $\cP$;
\item $\prod (d_is_B)$ is the product of the tangency orders at the nodes because each $f$ arises as a limit of that many distinct curves in the main stratum,
\item and, the product $\prod d_i$ in the denominator corresponds to the discrete automorphisms of $h_i$.
\eIt
Moreover, in (\ref{DOloc_e})
\bIt
\item $q=u_2(z_{2i})=u_1(z_{1i})$, for all $i=1,\ldots,k$, is the image of the nodes between $f_1$ and $f_2$ in $D$; 
\item   and, $\wt{H}^0(\P^1_2,u_2^*T\P_XD(-\log D_\infty \cup D_0))$ is the quotient of  ${H}^0(\P^1_2,u_2^*
T\P_XD(-\log D_\infty \cup D_0))$ by the complex 1-dimensional subspace corresponding to the $\C^*$-action on $\P_XD$.
\eIt 
Since
$$
u_2^* T\P_XD(-\log D_\infty \cup D_0)= \big(\P^1\times T_qD\big) \oplus \big(\P^1\times \C\big),
$$
we have 
$$
\wt{H}^0(\P^1_2,u_2^* T\P_XD(-\log D_\infty \cup D_0))\cong T_qD \quad\tn{and}\quad {
H}^1(\P^1_2,u_2^*T\P_XD(-\log D_\infty \cup D_0))=0.
$$
Since $s_B>1$, for all $i=1,\ldots,k$, the residue map
$$
\tn{Res}\colon H^0(\P^1_{1i},u_{1i}^*TX(-\log D))\lra  T_{q}D
$$
is zero. Also
 $$
 \wt{H}^0(\P^1_2,u_2^*
T\P_XD(-\log D_\infty \cup D_0))\cong T_qD
\stackrel{\tn{Res}}{\lra} \bigoplus_{i=1}^k T_{q}D
$$
is the diagonal embedding.
Therefore,
$$
\tn{Def}_{\log}(f)\cong H^0(\Si_1,u_1^*TX(-\log D))
$$
and 
$$
0\lra \frac{\bigoplus_{i=1}^k T_{q}D}{T_qD}\lra \tn{Ob}_{\log}(f) \lra  H^1(\Si_1,u_1^*TX(-\log D))\lra 0.
$$
Note that,
$$
H^1(\Si_1,u_1^*TX(-\log D))= \bigoplus_{i=1}^k H^1(\P^1_{1i},h_{1i}^*\cO(-1)).
$$
\noindent
Decomposing $\tn{Def}(C)$ into the moving and fixed factors, $\tn{Def}(C)=\tn{Def}(C)^{\tn{fix}}\oplus \tn{Def}(C)^{\tn{mov}}$, the moving part corresponds to the simultaneous smoothing of the nodes and has the equivariant contribution $-s_B -\psi$. Here, $-s_B$ is the weight of the action on $\cN_XD|_{q}$ and $\psi$ is the relative $\psi$-class in \cite[Sec.~2.5]{GV}. The fixed part $\tn{Def}(C)^{\tn{fix}} $ corresponds to the deformations of $\cM_{\cP}$. We conclude that the localization contribution $\tn{Cont}(\ov\cM_\cP)$ of $\cP$ is given by the formula
\bEq{Mmu-cont_e}
\aligned
\tn{Cont}(\ov\cM_\cP)&=\frac{s_B^{k}}{|\tn{Aut}(\cP)|} \Big(\prod_{i=1}^k d_i \tn{Cont}(f_{1i})\Big)(-1)^{k-1} \int_{\ov\cM_{\cP}} \frac{1}{-s_B-\psi}\\
&= \frac{s_B}{|\tn{Aut}(\cP)|} \Big(\prod_{i=1}^k d_i \tn{Cont}(f_{1i})\Big) \int_{\ov\cM_\cP} \psi^{k-2}\\
&=\frac{s_B}{|\tn{Aut}(\cP)|}  \prod_{i=1}^k \frac{(-1)^{d_i-1}}{d_i}{2d_i-1 \choose d_i-1} \int_{\ov\cM_\cP} \psi^{k-2}
\endaligned
\eEq
where $\tn{Cont}(f_{1i})$ is the contribution of the Galois cover $f_{1i}$ with $m=1$ in (\ref{m=1GC_e}) and $(-1)^{k-1}$ is the contribution of $\frac{\bigoplus_{i=1}^k T_{q}D}{T_qD}$ to $e\big(\tn{Ob}_{\log}(f)^{\tn{mov}}\big)$.  \\

\noindent
For instance, if $d=2$ and $\cP=\{1,1\}$, then $\ov\cM_{\cP}$ is a point, $\tn{Aut}(\cP)$ has order $2$, and 
$$
\tn{Cont}(\cM_\cP)=\frac{s_B}{2}.
$$
Together with the contribution of the Galois cover (i.e. $\cP=(2)$), which is $\frac{(-1)^{2-1}}{2^2}{2\times 2-1 \choose 2-1}=-3/4$, we
 get
 $$
\tn{mc}(2,s_B)= \frac{-3}{4}+\frac{2s_B}{4}=\frac{1}{2^2}{ 2(s_B-1)-1\choose 2-1}.
 $$

\noindent
If $d=3$, the partitions of $3$ are $\cP_1=(3), \cP_2=(2,1)$, and $\cP_3=(1,1,1)$.
Since $\ov\cM_{\cP_2}$ is a point, by (\ref{m=1GC_e}) and (\ref{Mmu-cont_e}), we have 
$$
\tn{Cont}(\cM_{\cP_1})=\frac{10}{9}, \quad \tn{Cont}(\cM_{\cP_2})=\frac{-3s_B}{2},
$$
Also, by \cite[Rmk.~3.4]{GV}, we have 
$$
\int_{\ov\cM_{\mu_2}} \psi= \frac{1}{3!}\int_{\ov\cM_{0,4}} 3s_B\psi_1=s_B/2;
$$
therefore, 
$$
\tn{Cont}(\cM_{\cP_3})=s_B^2/2.
$$
We conclude that 
$$
\tn{mc}(3,s_B)=\frac{10}{9}-\frac{3s_B}{2}+\frac{s_B^2}{2}=\frac{1}{3^2}{ 3(s_B-1)-1\choose 3-1}.
$$
\begin{remark}
The examples above suggest that the degree $k$ term of (\ref{mcdssB_e}), as a polynomial in $s_B$, is the sum of all $\tn{Cont}(\cM_{\cP})$ where $\cP$ is a partition of $d$ into $k+1$ summands.
\end{remark}

\section{Higher genus invariants with maximal tangency}\label{Gen_s}

Suppose $(X,D)$ is a symplectic log Calabi-Yau fourfold  and $g\geq 1$. By (\ref{dlog_e}), the real expected dimension of $\ov\cM_{g,(s_A)}(X,D,A)$ is $2g$. The relative GW invariants 
\bEq{NgA_e}
N_{g,A}= \int_{\big[\ov\cM^{\tn{rel}}_{g,(s_A)}(X,D,A)\big]^{\tn{VFC}}}(-1)^g \la_g \in \Q\qquad \forall~A\!\in\!H_2(X,\Z),~s_A>0,
\eEq
are defined by pairing the (algebraic) virtual fundamental class of $\ov\cM^{\tn{rel}}_{g,(s_A)}(X,D,A)$ with the degree $2g$ Hodge class $\la_g$. The latter is the top chern class of a rank $g$ (orbifold) vector bundle 
$$
\E\lra \ov\cM_{g,(s_A)}(X,D,A)
$$ 
whose fiber over $(u,\Si,z_1)$ with smooth domain is the space of holomorphic $1$-forms on $\Si$; more generally, $\E$ is defined to be the push-forward of the relative dualizing sheaf on the universal curve space $\cC_{g,(s_A)}(X,D,A)$.
These invariants are defined and studied in \cite[Sec.~5.8]{GPS} and \cite{B2}. 
The fact that the top lambda class is a natural insertion to consider in higher genus  stems from the fact that $N_{g,A}$ are 
related to the higher genus invariants of the threefold $X\times \P^1$, and (via localization) the top lambda class with the appropriate sign $(-1)^g$  measures of the difference between 2 and 3 dimensional obstruction theories; see \cite[p.13]{B2}.\\

\noindent
Let
\bEq{FAg_dfn}
F_A(q)=  \frac{\sin(h/2)}{h/2} \sum_{\substack{B\in H_2(X,\Z)\\ dB=A, d>0}} (-1)^{s_B-1}  \mu(d)~d^{2g-2} \sum_{g\geq 0}~N_{g,B}~h^{2g}
\eEq
where $\mu$ is the M\"obius function, $q=\tn{e}^{\mfi h/2}$, and $N_{0,A}=N_A$ are the genus $0$ invariants defined in (\ref{G0NA_e}).

\begin{conjecture}(c.f. \cite[Cnj.~8.3]{B2})\label{HgConj_cnj}
The function $F_A$ in (\ref{FAg_dfn}) is a well-defined rational function of $q$ invariant under $q\to q^{-1}$; furthermore,  it is a Laurent polynomial in $q$ 
with integer coefficients.
\end{conjecture}

\noindent
Note that, letting $h \lra 0$, only the genus zero terms survive, $\sin(h/2)$ cancels with $h/2$, and thus the value $F_A(1)$, which is the sum of the coefficients of the Laurent polynomial, is the genus $0$ BPS invariant
$$
\wt{n}_{A}= \sum_{\substack{B\in H_2(X,\Z)\\ dB=A, d>0}} \frac{\mu(d)}{d^2}  (-1)^{s_B-1} N_{B} \in \Z.
$$
The inverse of the formula above is 
$$
N_A=\sum_{\substack{B\in H_2(X,\Z)\\ dB=A, d>0}} \frac{(-1)^{s_B(d-1)}}{d^2} ~\wt{n}_{B}
$$
This formula looks different from (\ref{mc-formula_e}) but the integrality of  $\wt{n}_A$ is equivalent to the integrality of the genus $0$ BPS invariants $n_A$ in   (\ref{mc-formula_e}). A non-elementary but conceptual proof of this equivalence follows from Donaldson-Thomas theory of quivers.\\

\noindent
In the following, first, we explain how the higher genus invariants $N_{g,A}$ can be analytically defined without using virtual techniques at the cost of using (Ruan-Tian)-type global perturbations. Then, we revisit the unperturbed setup and explain the issue of multi-cover contributions.\\

\noindent
From the analytic point of view,  for $g\geq 1$, in order to construct a (virtual) fundamental class for $\ov\cM_{g,(s_A)}(X,D,A)$, we can  use the logarithmic/relative Ruan-Tian perturbations as in \cite{IP1,FT2} in the following way. For $g,k\in \N$ with $2g+k\geq 3$ (which is the case for all $g\geq 1$ and $k=1$), let $\cT_{g,k}$ denote the Teichm\"uler space of genus $g$ Riemann surfaces with $k$ marked points (punctures) and by $\mc{G}_{g,k}$ the corresponding mapping class group. We have
$$
\cM_{g,k}=\cT_{g,k}/ \mc{G}_{g,k}.
$$
Assume $g\!=\!g_1\!+\!g_2$ and $k\!=\!k_1\!+\!k_2$ with $2g_i\!+\!k_i\!\geq\!3$ for $i\!=\!1,2$. For any decomposition $S_1\cup S_2$ of $\{1,2,\ldots,k\}$ with $|S_i|=k_i$, there exists a canonical immersion 
\bEq{B1_e}
\iota=\iota_{S_1,S_2}\colon \ov\cM_{g_1,k_1+1}\times \ov\cM_{g_2,k_2+1}\lra \partial \ov\cM_{g,k}
\eEq
which assigns to a pair of marked  curves 
$$
\big(C_i\!=\![\Si_i,z_{i,1},\ldots,z_{i,k_i+1}]\big)_{i=1,2},
$$ 
the marked  curve
$$
\aligned
&C=[\Si,z_1,\ldots,z_k], \quad \Si\!=\! \Si_1\!\sqcup\! \Si_2/ z_{1,k_1+1}\!\sim \!z_{2,k_2+1}\\
&\{z_1,\ldots,z_k\}\!=\! \{z_{1,1},\ldots,z_{1,k_1}\}\cup  \{z_{2,1},\ldots,z_{2,k_2}\},
\endaligned
$$ 
so that the remaining marked points are renumbered by $\{1,\ldots,k\}$ according to the decomposition $S_1\!\cup\!S_2$. There is also another natural immersion 
\bEq{B2_e}
\de\colon \ov\cM_{g-1,k+2}\lra \partial \ov\cM_{g,k}
\eEq
which is obtained by gluing  together the last two marked points.

\begin{definition}[{\cite[Dfn.~2.1]{Z4}}]\label{UnivFamily_dfn}
Let $g,k\!\in\!\N$ with $2g\!+\!k\!\geq\!3$, and 
\bEq{Cover_e}
p\colon \ov{\mf{M}}_{g,k}\lra\ov\cM_{g,k}
\eEq
be a finite branched cover in the orbifold category.
A \textbf{universal family over} $ \ov{\mf{M}}_{g,k}$ is a tuple 
\bEq{UFamily_e}
\bigg(\pi\colon\ov{\mf{U}}_{g,k}\lra \ov{\mf{M}}_{g,k}, \mf{z}_1,\ldots,\mf{z}_k\bigg)
\eEq
where $\ov{\mf{U}}_{g,k}$ is a complex projective variety and $\pi$ is a projective morphism with disjoint sections $\mf{z}_1,\ldots,\mf{z}_k$ such that for each $c\!\in\!\ov{\mf{M}}_{g,k}$  the tuple 
$$
C\!=\!\big(\Si\!=\!\pi^{-1}(c),  (z_1,\ldots,z_k)=(\mf{z}_1(c),\ldots,\mf{z}_k(c))\big)
$$ 
is a stable $k$-marked genus $g$ curve  whose equivalence class is $[C]\!=\!p(c)$.
\eDf

\bDf{RefFamily_dfn}
Let $g,k\!\in\!\N$ with $2g\!+\!k\!\geq\!3$. A cover (\ref{Cover_e}) is \textbf{regular} if
\bEn
\item it admits a universal family,
\item each topological component of $p^{-1}\big(\cM_{g,k}\big)$ is the quotient of $\cT_{g,k}$ by a subgroup of $\mc{G}_{g,k}$,
\item\label{sepnode_l} for each boundary divisor (\ref{B1_e}) we have 
$$
\big(\ov\cM_{g_1,k_1+1}\times \ov\cM_{g_2,k_2+1}\big)\times_{(\iota,p)}\ov{\mf{M}}_{g,k} \approx \ov{\mf{M}}_{g_1,k_1+1}\times \ov{\mf{M}}_{g_2,k_2+1},
$$ 
for some regular covers $\ov{\mf{M}}_{g_i,k_i+1}$ of  $\ov\cM_{g_i,k_i+1}$, and
\item\label{selfnode_l} for the boundary divisor (\ref{B2_e}) we have 
$$
\ov\cM_{g-1,k+2}\times_{(\de,p)} \ov{\mf{M}}_{g,k} \approx \ov{\mf{M}}_{g-1,k+2},
$$ 
for some regular cover $\ov{\mf{M}}_{g-1,k+2}$ of $\ov\cM_{g-1,k+2}$.
\eEn
\eDf
\noindent
The last two conditions are inductively well-defined. 
The existence of such regular covers is a consequence of \cite[Prp.~2.2, Thm.~2.3, Thm.~3.9]{BoPi}; see also moduli space of curves with \textit{level $n$ structures} in \cite[p.~285]{Mu}. In the genus $0$ case, for each $k\!\geq\!3$, the moduli space $\ov\cM_{0,k}$ itself is smooth and the universal curve
$$
\ov\cC_{0,k}=\ov\cM_{0,k+1}\lra \ov\cM_{0,k}
$$  
is already a universal family. The regular covers are only branched over the boundaries of the moduli space. Furthermore, the total space of a universal family as in (\ref{UFamily_e}) over a regular cover only has singularities of the form
$$
\{(x,y,t)\in \C^3\colon~xy=t^m\} \lra \C, \qquad (x,y,t)\lra t
$$
at the nodal points of the fibers of $\pi$. In the original approach of \cite{RT}, for dealing with such singularities they consider embeddings of a universal family into $\P^N$ for  sufficiently large $N$. \\

\noindent
Fix a regular covering (\ref{Cover_e}) and a universal family (\ref{UFamily_e}). Denote by 
$$
\ov{\mf{U}}_{g,k}^\star\subset \ov{\mf{U}}_{g,k}
$$
the complement of the nodes of the fibers of the projection map $\pi$ in (\ref{UFamily_e}). Denote by 
$$
T_{g,k}= \tn{Ker}~\nd (\pi|_{\ov{\mf{U}}_{g,k}^\star}) \lra \ov{\mf{U}}_{g,k}^\star
$$
the vertical tangent bundle. The latter is a complex line bundle; we denote the complex structure by $\mfj_{\mf{U}}$. 
Then
$$
\Om^{0,1}_{g,k}:= (T_{g,k},-\mfj_{\mf{U}})^* \lra \ov{\mf{U}}_{g,k}^\star
$$
is the complex line bundle of  vertical (or relative) $(0,1)$-forms. It is possible to extend this construction to the nodal points by allowing simple poles and dual residues, or by embedding $\ov{\mf{U}}_{g,k}$ into some $\P^M$ as in \cite{RT}. \\

\noindent
Let $(X,\om)$ be a  symplectic manifold and $J$ be an $\om$-tame almost complex structure on $X$.  The classical space of perturbations considered in \cite{RT} (following the modification in \cite{Z4}) is the infinite dimensional linear space
\bEq{cHgk_e}
\cH_{g,k}(X,J)\!=\!\big\{ \nu\! \in\! \Gamma\big(\ov{\mf{U}}_{g,k}^\star\times X, \pi_1^*\Om^{0,1}_{g,k}\otimes_\C \pi_2^*TX\big)~~\tn{s.t.}~~\tn{supp}(\nu)\!\subset(\ov{\mf{U}}_{g,k}^\star-\bigcup_{a=1}^k \tn{Im}(\mf{z}_a))\!\times\! X \big\},
\eEq
where $\pi_1,\pi_2$ are projection maps from $\ov{\mf{U}}_{g,k}^\star\times X$ onto the first and second components, respectively, and $\tn{supp}(\nu)$ is the closure of the complement of the vanishing locus of $\nu$ in the compact space $\ov{\mf{U}}_{g,k}\times X$. Let $\cH_{g,k}(X,\om)$ denote the space of tuples $(J,\nu)$ where $J$ is $\om$-tame and $\nu\!\in\!\cH_{g,k}(X,J)$. Note that given $\nu$ and a boundary component as in Definition~\ref{RefFamily_dfn}.\ref{sepnode_l} (resp.  Definition~\ref{RefFamily_dfn}.\ref{selfnode_l}), the restriction of $\nu$ to $ \ov{\mf{M}}_{g_1,k_1+1}$ gives a perturbation term in $\cH_{g_1,k_1}(X,J)$ (resp. $\cH_{g-1,k+2}(X,J)$).

\bDf{JnuModuli_dfn}
Suppose $g,k\!\in\!\N$ with $2g\!+\!k\!\geq\!3$, $\ov{\mf{U}}_{g,k}$ is a universal family as in (\ref{UFamily_e}), $(X,\om)$ is a symplectic manifold, $A\!\in\!H_2(X,\Z)$, and $(J,\nu)\!\in\!\cH_{g,k}(X,\om)$. A \textbf{degree $A$ genus $g$ $k$-marked $(J,\nu)$-map} is a tuple 
\bEq{JnuMap_e}
f=\Big( \phi ,u, C=\big(\Si, (z_1,\ldots,z_k)\big) \Big)
\eEq
where $C$ is a nodal genus $g$  $k$-marked complex curve, $\phi\colon\!\Si\! \lra\! \ov{\mf{U}}_{g,k}$ is a holomorphic map onto a fiber of $\ov{\mf{U}}_{g,k}$ preserving the marked points,  and $u\colon\!\Si\!\lra\!X$ represents the homology class $A$ and satisfies 
$$
\dbar u =(\phi,u)^*\nu.
$$
\eDf
\noindent
Two $k$-marked $(J,\nu)$-holomorphic maps $\big( \phi_1 ,u_1, C_1\big)$ and $\big( \phi_2 ,u_2,C_2\big) $
are \textbf{equivalent} if there exists a holomorphic identification $h$ of $C_1$ and $C_2$
such that $(\phi_1,u_1)\!=\!(\phi_2,u_2)\circ h$. A $(J,\nu)$-holomorphic map is \textbf{stable} if it has a finite automorphism group. A \textbf{contracted} component of $\Si$ in (\ref{JnuMap_e}) is a smooth component whose image under the map $\phi$ is just a point. 
A map (\ref{JnuMap_e}) is stable  if and only if the degree of the restriction of $u$ to every contracted component of $\Si$ containing only one or two special (nodal or marked) points is not zero.
If (\ref{JnuMap_e}) is stable, every connected cluster of contracted components is a tree of spheres, with a total of at most $2$ special\footnote{either a marked point or a nodal point connecting the cluster to an irreducible component of $\Si$ outside the cluster.} points, at least one of which is a nodal point. 
For generic $\nu$, the only components of $\Si$ contributing non-trivially to the automorphism group of (\ref{JnuMap_e}) are the contracted components.\\

\noindent
In order to perturb $J$-holomorphic curves with tangency condition relative to a (smooth or normal crossing) divisor $D\subset X$, we need to restrict to subset of $\cH_{g,k}(X,J)$ consisting of perturbations that are ``compatible" with $D$ in suitable sense. In \cite{IP1,FZ1}, the latter is expressed in terms of a first order condition on $\nu$ along $D$. In \cite{FT2}, we consider the logarithmic perturbation space
\bEq{cHgk_e2}
\bigg\{ \nu_{\log}\! \in\! \Gamma\big(\ov{\mf{U}}_{g,k}^\star\times X, \pi_1^*\Om^{0,1}_{g,k}\otimes_\C \pi_2^*TX(-\log D)\big)~~\tn{s.t.}~~\tn{supp}(\nu_{\log})\!\subset\big(\ov{\mf{U}}_{g,k}^\star-\bigcup_{a=1}^k \tn{Im}(\mf{z}_a)\big)\!\times\! X \bigg\}.
\eEq
Associated to each $\nu_{\log}$ we get a classical perturbation term 
\bEq{AsscNu_e}
\nu\!=\!\iota(\nu_{\log})\!\in\! \cH_{g,k}(X,J),
\eEq
where by abuse of notation $\iota$ denotes the $\C$-linear homomorphism  
$$
\pi_1^*(\Om^{0,1}_{g,k})\otimes_\C \pi_2^*TX(-\log D)\lra \pi_1^*(\Om^{0,1}_{g,k})\otimes_\C \pi_2^*TX
$$
induced by (\ref{iota_e}). Define $\cH_{g,\mfs}(X,D)$ to be the space of tuples $(\om,J,\nu_{\log})$ where $(\om,J)\in \tn{AK}(X,D)$ and $\nu_{\log}$ belongs to (\ref{cHgk_e2}). For such a tuple $(\om,J,\nu_{\log})$, it is shown in \cite{FT2} that the moduli space $\cM_{g,\mfs}(X,D,A,\nu)$ consisting of equivalence classes of genus $g$ $k$-marked $(J,\nu)$-holomorphic maps of tangency type $\mfs$ is defined and has a natural compactification $\ov\cM^{\log}_{g,\mfs}(X,D,A,\nu)$ satisfying properties similarly to the unperturbed case. The relative moduli space $\ov\cM^{\tn{rel}}_{g,\mfs}(X,D,A,\nu)$   can be defined similarly and admits a surjective map 
$$
\ov\cM^{\tn{rel}}_{g,\mfs}(X,D,A,\nu)\lra \ov\cM^{\log}_{g,\mfs}(X,D,A,\nu)
$$
as in the unperturbed case.
\\

\noindent
If $(X,D)$ is a symplectic log CY fourfold, then it is semi-positive in the sense of \cite[Dfn.~4.7]{FZ1} or equally \cite[Dfn.~1.6]{FT2}. Therefore, by  \cite[Prp.~1.7 and Crl.~1.9]{FT2}, we have the following result.

\bTh{TransGamma_thm} 
Suppose $(X,D)$ is a symplectic log CY fourfold,  $A\!\in\!H_2(X,\Z)$, and $g>0$. 
Then, for any given choice of universal family in (\ref{UFamily_e}), there exists a Baire set of second category $\cH^{\tn{reg}}_{g,(s_A)}(X,D)\!\subset\!\cH_{g,(s_A)}(X,D)$ such that for each $(\om,J,\nu_{\log})\!\in\!\cH^{\tn{reg}}_{g,(s_A)}(X,D)$, \bEn
\item the moduli space $\cM^{\star}_{g,(s_A)}(X,D,A)$ is cut transversely and is a smooth manifold of real dimension 
$2g$,
\item the image of  $\ov\cM_{g,(s_A)}(X,D,A)-\cM^{\star}_{g,(s_A)}(X,D,A)$ under the forgetful map
\bEq{st_e}
\tn{st}\colon\!\ov\cM_{g,(s_A)}(X,D,A)\!\lra \ov\cM_{g,1}
\eEq
lies in the image of smooth maps from finitely many smooth even-dimensional manifolds of at least $2$ real dimension less than $2g$,
\item and consequently, the map (\ref{st_e}) defines a pseudo-cycle of real dimension $2g$ in $\ov\cM_{g,1}$ whose integral homology class $[\tn{st}]$ only depends on the deformation equivalence class of  $(X,D,\om)$ and the degree $\tn{deg}~p$ of the regular covering  (\ref{Cover_e}) used to define $\nu_{\log}$.
\eEn
\eTh

\begin{corollary}
The rational homology class
$$
\big[\ov\cM_{g,(s_A)}(X,D,A)\big]^{\tn{VFC}}\equiv \frac{1}{\tn{deg}~p}[\tn{st}]\in H_{2g}(\ov\cM_{g,1}, \Q),
$$
is an invariant of the deformation equivalence class of $(X,\om,D)$ and can be used to define (\ref{NgA_e}). Furthermore, for every $g>0$, there exists a constant $c_g$ (independent of the choice of $(X,D)$ and $A$) such that 
$$
c_g N_{g,A}\in \Z \qquad \forall A\in H_2(X,\Z),~\tn{with}~s_A>0.
$$
\end{corollary}

\noindent
Using the perturbed setting above is an effective method for defining the invariants $N_{g,A}$ but it is not useful for defining integer-valued invariants and understanding the geometric meaning of the integers predicted by Conjecture~\ref{HgConj_cnj}. For $g>0$, given a $1$-marked smooth curve $(\Si,p)$, the relative space $\ov\cM_{g,(d)}(\Si,p,[d\Si])$ of genus $g$ degree $d$ covers fully ramified over $p$ is not necessarily a smooth orbifold of the correct complex dimension. Furthermore, the following observations illustrate the complicated nature of relative multiple-cover maps in higher genus.\\

\noindent
First, we have the following weaker analogue of Lemma~\ref{LItoSI_lmm2} which is a corollary of \cite[Thm~1']{HLS}.

\bLm{LItoSI_lmm2}
Suppose $(X,D)$ is a symplectic log Calabi-Yau fourfold and 
$$
f\!=\![u,\Si,z_1]\!\in\! \cM_{g,(s_A)}(X,D,A).
$$
If $f$ is a log immersion then $\tn{deg}(\cN_f)=2g-1$,
$$
\Def_{\log}(f)=H^0(\cN_f)\quad\tn{and}\quad  \Ob_{\log}(f)=H^1(\cN_f)= 0,
$$
where $\cN_f$ is the logarithmic normal bundle defined in (\ref{LogNorBundle_e}). In other words, $\cM_{g,(s_A)}(X,D,A)$ is cut  transversally in a neighborhood of every log immersion.
\eLm

\noindent
In \cite{R-S}, given an elliptic curve $\Si$ and an integer $g\geq 2$, the author constructs a genus $g$ curve $\Si'$ and a map $h\colon \Si'\lra \Si$ of degree $2g-1$, so that $h$ is ramified above exactly one point of $\Si$, and so that the local monodromy above that point is of type a $(2g-1)$-cycle. Together with Lemma~\ref{LItoSI_lmm2}, we get the following observation. Suppose $u\colon (\Si,p)\lra (X,D)$ is a genus $1$ log immersion in $\cM_{1,(s_A)}(X,D,A)$ and $h\colon (\Si',z_1)\lra (\Si,p)$ is a degree $2g-1$ map ramified at exactly $p$. Then, by Lemma~\ref{LItoSI_lmm2}, the moduli space $\cM_{g,(s_A)}(X,D,A)$ is cut transversely in a neighborhood of the multiple-cover map $u\circ h\colon (\Si',z_1)\lra (X,D)$. Therefore, $u$ can not have a multiple-cover contribution to an integral-valued invariant arising from $\ov\cM_{g,(s_A)}(X,D,A)$.  
Nevertheless, for $g=1$, the possible multiple-covers are well-understood and moduli space $\ov\cM_{1,(s_A)}(X,D,A)$ and the invariant $N_{1,A}$ can be explicitly described as follows.\\

\noindent
For $A\neq 0$, let $\cM'_{1,(s_A)}(X,D,A)$ denote the subset of $\ov\cM_{1,(s_A)}(X,D,A)$ consisting of the relative/log stable maps $[u,C=(\Si,z_1)]$ such that $\Si$ is an elliptic curve $E$ with $1$ rational component attached directly to it and $u|_{E}$ is constant (therefore, the restriction of $u$ to the rational component has degree $A$). In other words, 
$$
\cM'_{1,(s_A)}(X,D,A)\cong \cM_{0,(0,s_A)}(X,D,A)\times \cM_{1,1}.
$$
We denote by $\ov\cM'_{1,(s_A)}(X,D,A)$ the closure of $\cM'_{1,(s_A)}(X,D,A)$ in $\ov\cM_{1,(s_A)}(X,D,A)$. 
By \cite[Thm~1.2]{Z1}, the moduli space $\ov\cM_{1,(s_A)}(X,D,A)$ admits a closed subspace $\ov\cM_{1,(s_A)}^{\tn{main}}(X,D,A)$ that contains the virtually main stratum $\cM_{1,(s_A)}(X,D,A)$ and is invariant under deformations of $J$. Furthermore, if $(X,D)$ is a log Calabi-Yau fourfold and $J$ is generic, the moduli space $\ov\cM_{1,(s_A)}(X,D,A)$ decomposes as 
\bEq{g1Dec_e}
 \ov\cM_{1,(s_A)}(X,D,A)=\ov\cM_{1,(s_A)}^{\tn{main}}(X,D,A)\cup \ov\cM'_{1,(s_A)}(X,D,A).
\eEq
By (\ref{g1Dec_e}), we have 
$$
N_{1,A}=N_{1,A}^{\tn{main}}+N_{1,A}'
$$
such that $N_{1,A}'$  is a function of $\{n_B\}_{dB=A}$ determined by the lemma below. 
\begin{lemma}
We have 
$$
N_{1,A}'= \sum_{\substack{dB=A\\ d>0}}\frac{(2-s_B)}{24}{d(s_B-1)-1 \choose d-1} n_{B} \in \frac{1}{24}\Z.
$$
\end{lemma}
\begin{proof}
For the same choice of weight as in (\ref{mcdssB_e}), the contribution of a logarithmically immersed rational curve $[\ov{u},(\P^1,\infty)]\in \cM^\star_{0,(s_B)}(X,D,A)$ to $N_{1,A=dB}'$ has the form 
\bEq{g=1integral}
\int_{ \ov\cM_{1,1}}  \Big( \frac{1}{d} \frac{e(\tn{Ob}_{\log}(f)^{\tn{mov}})e(\tn{aut}(C)^{\tn{mov}})}{e(\tn{Def}(f)^{\tn{mov}}) e(\tn{Def}_{\log}(C)^{\tn{mov}})}\times  (-\la_1)\Big)
\eEq
where $C=(E\cup_{q\sim 0} \P^1, \infty)$ is the $1$-marked nodal domain obtained by attaching $\P^1$ and an elliptic curve $E$ at the points $0$ and $q$, respectively, with the marked point $\infty$ on the rational component. Over $\P^1$, the map $f$ is the $d$-fold Galois cover $(\P^1,0\cup \infty)\lra (\P^1, p_0\cup p_1)$, and over $E$ the map is constant $p_0$.
Compared to the calculations leading to (\ref{LCm_e}), the node in the domain contributes a factor of 
$$
\frac{1}{\frac{1}{d}-\psi_1}= d \frac{1}{1-d\psi_1} 
$$
to the moving part $\tn{Def}(C)^{\tn{mov}}$ and the automorphism factor $\tn{aut}(C)$ of $(\P^1,0\cup \infty)$ is the weight zero representation
$$
H^0(\P^1,T\P^1(-\log  0\cup \infty))=\C \cdot (z_0\partial z_0).
$$ 
These two changes each contribute an extra factor of $d$ to the resulting fraction in (\ref{LCm_e}).
Since the weights on the tangent space of the (standard) normal bundle to $\ov{u}$ at $p_0$ are $s_B-2$ and $1$, by \cite[p.~550]{Mir}, the obstruction bundle over $E$ contributes a factor of 
\bEq{ObsofE_e}
(s_B-2)~c(\E^\vee)(1/(s_B-2))~c(\E^\vee)(1)
\eEq
to $e(\tn{Ob}_{\log}(f)^{\tn{mov}}/e(\tn{Def}_{\log}(f)^{\tn{mov}})$ where $\E$ is the Hodge bundle over $E$ and 
$$
c(Q)(t) = 1 + tc_1(Q) + \cdots + t^rc_r(Q)
$$
for any complex vector bundle $Q$ of rank $r$. By dimensional reason, only the degree zero term $s_B-2$ of (\ref{ObsofE_e}) contributes to the integral (\ref{g=1integral}).
Since 
$$
\int_{\ov\cM_{1,1}}\la_1=\frac{1}{24},
$$
we conclude that (\ref{g=1integral}) is equal to 
$$
\frac{(2-s_B)}{24}{d(s_B-1)-1 \choose d-1}.
$$
\end{proof}

\begin{remark}
For $s_B=1$, we get 
$$
\frac{(2-s_B)}{24}{d(s_B-1)-1 \choose d-1}= \frac{(-1)^{d-1}}{24}
$$
which is the number calculated in \cite[(13)]{BP}.
\end{remark}

\noindent
\noindent
The reduced moduli space $\ov\cM^{\tn{main}}_{1,(s_A)}(X,D,A)$ may still include components of higher than expected dimension. For instance, if $A=dB$, $\ov\cM^{\tn{main}}_{1,(s_A)}(X,D,A)$ includes multiple-covers of the form
$$
u\colon (E,q)\stackrel{h}{\lra} (\P^1,\infty) \stackrel{\ov{u}}{\lra} (X,D),
$$
where $h\in \cM_{1,(d)}(\P^1,\infty,[d])$ is a genus $1$ multiple-cover of $\P^1$ fully ramified at $\infty$. Note that 
$$
\dim_\C \cM_{1,(d)}(\P^1,\infty,[d])=d+1.
$$
However, since the class $\la_1$ vanishes on the locus of curves with loops (c.f. \cite[p.~351]{GPS}), the (localization) contribution of such multiple-covers to $N_1^{\tn{main}}$ is zero. Also, unlike in classical GW theory, isogenies of elliptic curve do not appear because of the maximal tangency condition.
Also, note that by fixing the complex structure of the domain elliptic curve, we can reduce the complex dimension to zero and avoid the integration of $\la_1$ at the cost of multiplying by the constant 
$$
\tn{PD}_{\ov\cM_{1,1}}(\la_1)\in H_0(\ov\cM_{1,1},\Q)\cong \Q.
$$ 
Therefore, we expect the reduced genus-one GW-invariants $N_{1,A}^{\tn{main}}$ arising from $\ov\cM_{1,(s_A)}^{\tn{main}}(X,D,A)$ to be a fixed multiple of an integer count of genus one maximally-tangent logarithmically immersed curves in $(X,D)$ with a fixed complex structure on the domain.

\appendix 
\section{Relative Compactification}\label{RC_A}

In this section, following the description in \cite{FZ1}, we briefly review the construction of the relative stable map moduli spaces $\ov\cM_{g,\mfs}^{\tn{rel}}(X,D,A)$ and clarify the notation used in the rest of the paper. For more details, we refer to \cite{FZ1,IP1,FT1}.\\

\noindent
Suppose $D\!\subset\!(X,\om)$ is a smooth symplectic divisor, $J$ is an $\om$-tame almost complex structure on $X$ compatible with $D$, and $\dbar_{\cN_XD}$ is the $\dbar$-operator arising from $J$ on the normal bundle 
\bEq{Firstpi_e}
\pi\colon \cN_XD\equiv \frac{TX|_{D}}{TD}\lra D.
\eEq
Let
\bEq{PXV_e}
\P_X D= \P(\cN_X D\oplus D\!\times\!\C),\qquad D_{0}= \P(0\oplus D\!\times\!\C)~~\tn{and}~~ D_{\infty}=\P(\cN_X D\oplus 0) \subset \P_XD.
\eEq
The splitting (\ref{Tidentification_e}) extends to a splitting of the exact sequence 
$$
0\lra T^{\tn{vrt}}(\P_XD) \lra T(\P_XD)
\stackrel{\nd\pi}{\lra} \pi^*TD\lra0,
$$
where $\pi\!:\P_XD\!\lra\!D$ is the bundle projection map induced by $\pi$ in (\ref{Firstpi_e}); this splitting restricts to the canonical splittings over $D_{0}\!\cong\!D_{\infty}\!\cong\! D$
and is preserved by the multiplication by~$\C^*$.
Via this splitting, the almost complex structure $J_D=J|_{TD}$ and 
the complex structure $\mfi$ in the fibers of~$\pi$ induce
an almost complex structure~$J_{X,D}$ on~$\P_XD$ which restricts to $J_D$ on $D_{0}$ and~$D_{\infty}$ and is preserved by the~$\C^*$-action. 
In fact, $J_{X,D}|_{\cN_XD}$ is the almost complex structure $J_{\dbar_{\cN_XD}}$ associated to $\dbar_{\cN_XD}$.
The projection $\pi\!:\P_XD\!\lra\!D$ is $(J_D,J_{X,D})$-holomorphic and there is a one-to-one correspondence between 
the space of $J_{X,D}$-holomorphic maps $u\colon\!\Si\!\lra\!\P_X D$ and tuples $(u_D,\ze)$ where $u_D\colon \Si\!\lra\!D$ is a  $J_D$-holomorphic map into $D$ and $\ze$ is a meromorphic section of $u_D^*\cN_XD$ with respect to the holomorphic structure defined by $u^*\dbar_{\cN_XD}$ on $u^*\cN_XD$.\\

\noindent
For each $\ell\!\in\!\N$, the $\ell$-th expanded degeneration of $X$ is the normal crossings variety 
$$
X[\ell]=\big(X\sqcup\{1\}\!\times\!\P_XD\sqcup\ldots\sqcup
\{\ell\}\!\times\!\P_XD\big)/\!\!\sim\,,
$$
where
$$
D\sim \{1\}\!\times\!D_{\infty}\,,~~~
\{r\}\!\times\!D_{0}\sim \{r\!+\!1\}\!\times\!D_\infty
\quad  \forall~r\!=\!1,\ldots,\ell\!-\!1;
$$
see Figure~\ref{relcurve_fig}.
There exists a continuous projection map $\pi_\ell\colon\!X[\ell]\!\lra\!X$ which is identity on $X$ and $\pi$ on each $\P_{X}D$.
We denote by $J_\ell$ the almost complex structure on~$X[\ell]$ so that 
$$
J_\ell|_X=J_X \qquad\hbox{and}\qquad 
J_\ell|_{\{r\}\times\P_XD}=J_{X,D} \quad \forall~r=1,\ldots,\ell.$$
For each $(t_1,\ldots,t_\ell)\!\in\!\C^*$, define
$$
\Theta_{t_1,\ldots,t_\ell}\!:X[\ell]\lra X[\ell] \qquad\hbox{by}\quad
\Theta_{t_1,\ldots,t_\ell}(x)=\begin{cases}(r,[t_rv,w]),&\hbox{if}~x\!=\!(r,[v,w])\!\in\!\{r\}\!\times\!\P_XD;\\x,&\hbox{if}~x\!\in\!X.
\end{cases}
$$
This diffeomorphism is biholomorphic with respect to~$J_\ell$ and
preserves the fibers of the projection $\P_XD\!\lra\!D$
and the sections~$D_{0}$ and~$D_{\infty}$.\\

\noindent
For $\ell> 0$, a level $\ell$ $k$-marked \textbf{relative} $J$-holomorphic map of contact type~$\mfs\!=\!(s_1,\ldots,s_{k})\!\in\!\N^{k}$ is
a continuous map $u\!:\Si\!\lra\!X[\ell]$ 
from a connected marked  nodal curve $C=\big(\Si,(z_1,\ldots,z_k)\big)$ 
such~that 
$$
u^{-1}\big(\{\ell\}\!\times\!D_0\big)
\subset\big\{z^{1},\ldots,z^{k}\big\}, \quad
\tn{ord}_{z_a}(u,\{\ell\}\!\times\!D_0)\!=\!s_a\quad \forall~z_a\!\in\!u^{-1}\big(\{\ell\}\!\times\!D_0\big),
$$
 $s_a\!=\!0$ if and only if $z_a\!\notin\!u^{-1}\big(\{\ell\}\!\times\!D_0\big)$, and the restriction of~$u$ to each irreducible component $\Si_j$ of~$\Si$ is either 
\bEn
\item a $J$-holomorphic map to $X$ such that the set $u|_{\Si_j}^{\,-1}(D)$ 
consists~of the nodes joining~$\Si_j$ to irreducible components of~$\Si$ mapped to
$\{1\}\!\times\!\P_XD$, or
\item a $J_{X,D}$-holomorphic map to $\{r\}\!\times\!\P_XD$ for some $r\!=\!1,\ldots,\ell$ such~that 
\bEnalph
\item the set $u|_{\Si_j}^{\,-1}(\{r\}\!\times\!D_{\infty})$ 
consists~of the nodes~$q_{j,i}$ joining~$\Si_j$ to irreducible components of~$\Si$ mapped to
$\{r\!-\!1\}\!\times\!\P_XD$ if $r\!>\!1$ and to~$X$ if $r\!=\!1$
and 
$$
\tn{ord}_{q_{j,i}}\big(u,D_{\infty}\big)=
\begin{cases}
\tn{ord}_{q_{i,j}}(u,D_{0}),&\hbox{if}~r\!>\!1,\\
\tn{ord}_{q_{i,j}}(u,D),&\hbox{if}~r\!=\!1,
\end{cases}
$$ 
where $q_{i,j}\!\in\!\Si_{i}$ is the point identified with~$q_{j,i}$, 
\item  if $r\!<\!\ell$, the set $u|_{\Si_j}^{\,-1}(\{r\}\!\times\!D_{0})$ 
consists~of the nodes joining~$\Si_j$ to irreducible components of~$\Si$ mapped to
$\{r\!+\!1\}\!\times\!\P_XD$;
\eEnalph
\eEn
see Figure~\ref{relcurve_fig}.
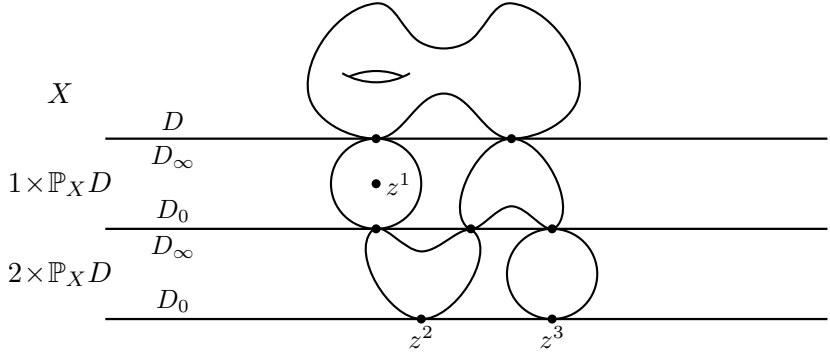
\begin{figure}
\begin{pspicture}(38,-1)(11,3.2)
\psset{unit=.3cm}
\psline[linewidth=.1](48,-2)(80,-2)
\psline[linewidth=.1](48,2)(80,2)
\psline[linewidth=.1](48,6)(80,6)

\rput(51,6.8){\small{$D$}}\rput(51,5.2){\small{$D_{\infty}$}}
\rput(51,2.8){\small{$D_{0}$}}\rput(51,1.2){\small{$D_{\infty}$}}
\rput(51,-1.2){\small{$D_{0}$}}

\pscircle*(62,-2){.2}\rput(62,-2.8){\small{$z^2$}}
\pscircle*(67.8,-2){.2}\rput(67.8,-2.8){\small{$z^3$}}

\rput(46,8){$X$}\rput(46,4){$1\!\times\!\P_XD$}\rput(46,0){$2\!\times\!\P_XD$}

\psccurve(63,10)(62,10.25)(60,12)(57,8)(60,6)(63,8)(66,6)(69,8)(66,12)(64,10.25)\pscircle*(60,6){.2}\pscircle*(66,6){.2}

\pscircle*(60,2){.2}\pscircle*(64.2,2){.2}\pscircle*(67.8,2){.2}\pscircle*(60,4){.2}
\pscircle(60,4){2}
\rput(61,4){\small{$z^1$}}
\pscircle(67.8,0){2}
\psccurve(64,2.1)(66,3)(68,2.1)(66,6)
\psccurve(59.8,1.9)(62,1)(64.4,1.8)(62,-2)

\psarc(60,6){3}{67}{113}
\psarc(60,11.5){3}{240}{300}

\end{pspicture}
\caption{A genus $2$ relative map with $k\!=\!3$ and $\mfs\!=\!(0,2,2)$
into the expanded degeneration $X[2]$.}
\label{relcurve_fig}
\end{figure}
The \textbf{genus} and the \textbf{degree} of such a map $u\!:\!\Si\!\lra\!X[\ell]$
are the arithmetic genus of~$\Si$ and the homology class
$$
A=\big[\pi_\ell\!\circ\!u\big]\in H_2(X,\Z).
$$

\noindent
Two tuples $(u,C)$ and $(u',C')$  as above are \textbf{equivalent}
if there exist a biholomorphic map $\varphi\colon\Si\!\lra\!\Si'$ 
and \hbox{$t_1,\ldots,t_\ell\!\in\!\C^*$} so~that 
$$
\varphi(z_a)=z'_a \quad \forall~a\!=\!1,\ldots,k \quad\hbox{and}\quad
u'=\Theta_{t_1,\ldots,t_\ell}\circ u \circ\varphi.
$$
A tuple as above is \textbf{stable} if it has finitely many automorphisms (self-equivalences).\\

\noindent
If $A\!\in\!H_2(X,\Z)$, $g,k\!\in\!\N$, and $\mfs\!=\!(s_1,\ldots,s_{k})\!\in\!\N^{k}$ 
is a tuple satisfying
$$
\sum_{a=1}^{k} s_a = A\cdot D,
$$ 
then the \textbf{relative moduli space} 
\bEq{relmoddfn_e}
\ov\cM_{g,\mfs}^{\tn{rel}}(X,D,A)
\eEq
is the set of equivalence classes of such connected stable $k$-marked genus~$g$ degree~$A$ $J$-holomorphic maps
into  $X[\ell]$ for any $\ell\!\in\!\N$. If $X$ is compact, the latter space has a natural compact Hausdorff topology with respect to which the forgetful map 
$$
\ov\cM_{g,\mfs}^{\tn{rel}}(X,D,A)\lra \ov\cM_{g,k}(X,A)
$$
is continuous. The logarithmic compactification $\ov\cM_{g,\mfs}^{\tn{log}}(X,D,A)$ defined in \cite{FT1} replaces the level structure with a partial ordering on the components of the domain and collapses unstable components that are maps from $\P^1$ with less than three special points to fibers of $\P_XD$.

\vspace{.2in}

\noindent
{\it The University of Iowa, MacLean Hall, Iowa City, IA 52241\\
mohammad-tehrani@uiowa.edu}\\

\end{document}